\documentclass[12pt, a4paper]{article}

\usepackage[utf8]{inputenc}
\usepackage[T1]{fontenc}
\usepackage{amsmath,amsthm,amssymb, multicol, array, mathtools}
\usepackage{stmaryrd}
\usepackage{geometry}
\usepackage{graphicx}
\usepackage{subcaption}
\usepackage{enumitem}
\usepackage{appendix}
\usepackage{titling}
\usepackage{hyperref}
\usepackage[backend=bibtex]{biblatex}

\newcommand{\res}{\mathbin{\vrule height 1.6ex depth 0pt width
0.13ex\vrule height 0.13ex depth 0pt width 1.3ex}}
\newcommand{\Alt}{\operatorname{\raisebox{0.3ex}{\footnotesize $\bigwedge$}}\!}
\newcommand{\bb}[1]{\llbracket #1\rrbracket} 
\theoremstyle{theorem}
\newtheorem{theorem}{Theorem}[section]
\newtheorem{proposition}[theorem]{Proposition}
\newtheorem{corollary}[theorem]{Corollary}
\newtheorem{lemma}[theorem]{Lemma}
\newtheorem*{claim*}{Claim}
\newtheorem{claim}{Claim}

\theoremstyle{definition}
\newtheorem{definition}[theorem]{Definition}
\newtheorem{example}[theorem]{Example}

\theoremstyle{remark}
\newtheorem{remark}[theorem]{Remark}

\newcommand{\sF}{\mathcal{F}}
\newcommand{\sN}{\mathcal{N}}

\newcommand{\sH}{\mathcal{H}}
\newcommand{\sP}{\mathcal{P}}

\newcommand{\ssm}{\smallsetminus}

\newcommand{\R}{\mathbb{R}}

\newcommand{\N}{\mathbb{N}}

\newcommand{\CAT}{\operatorname{CAT}}

\newcommand{\can}{\operatorname{can}}

\newcommand{\Lip}{\operatorname{Lip}}

\newcommand{\Mloc}[1]{\mathbf{M}_{#1,\mathrm{loc}}}
\newcommand{\sRloc}[1]{\mathcal{R}_{#1,\mathrm{loc}}}
\newcommand{\sIloc}[1]{\mathcal{I}_{#1,\mathrm{loc}}}
\newcommand{\Nloc}[1]{\mathbf{N}_{#1,\mathrm{loc}}}
\newcommand{\Iloc}[1]{\mathbf{I}_{#1,\mathrm{loc}}}
\newcommand{\Zloc}[1]{\mathbf{Z}_{#1,\mathrm{loc}}}

\newcommand{\Mak}[1]{\mathbf{M}_{#1}}
\newcommand{\sRak}[1]{\mathcal{R}_{#1}}
\newcommand{\sIak}[1]{\mathcal{I}_{#1}}
\newcommand{\Nak}[1]{\mathbf{N}_{#1}}
\newcommand{\Iak}[1]{\mathbf{I}_{#1}}
\newcommand{\Zak}[1]{\mathbf{Z}_{#1}}

\newcommand{\Mass}{\mathbf{M}}
\newcommand{\spt}{\operatorname{spt}}
\newcommand{\set}{\operatorname{set}}

\newcommand{\FillVol}{\operatorname{FillVol}}
\newcommand{\diam}{\operatorname{diam}}

\title{On the asymptotic Plateau problem for cycles in the Tits boundary of a Hadamard space}
\author{Hjalti Isleifsson}
\date{\vspace{-48pt}}

\begin{document}

\allowdisplaybreaks

\maketitle

\makeatletter{\renewcommand*{\@makefnmark}{}
\footnotetext{{\it Date}: December 12, 2023.}
\footnotetext{Research supported by Swiss National Science Foundation Grant 197090.}
\makeatother}

\begin{abstract}
The dimension of cycles in the Tits boundary of a proper Hadamard space is bounded by the asymptotic rank \(m\) of the space minus one. Kleiner and Lang proved that for \((m-1)\)-dimensional cycles in the Tits boundary, the asymptotic Plateau problem can be solved. We prove that the asymptotic Plateau problem can be solved for so called {\it strongly immovable} cycles in the Tits boundary of a proper Hadamard space. The class of strongly immovable cycles contains all \((m-1)\)-cycles. Further, it is a result of Huang, Kleiner and Stadler that the class of strongly immovable cycles contains the boundaries of cocompact flats which do not bound flat half-spaces and we prove that the class of strongly immovable cycles of a fixed dimension forms a group so our result applies to new examples.
\end{abstract}

\section{Introduction}

In the context of Hadamard spaces (that is, complete \(\CAT(0)\) spaces) – or more generally, spaces satisfying some sort of a non-positive curvature condition – by an asymptotic Plateau problem is usually meant the study of existence of a minimal variety which is asymptotic in some sense to a given variety in a boundary at infinity. Two important boundaries of study are the {\it visual boundary} and the {\it Tits boundary}; these boundaries agree as sets – and we will refer to their underlying set as {\it the boundary at infinity} – but the visual boundary is endowed with the so called cone-topology while the Tits boundary is endowed with the Tits metric which induces a finer topology (cf. \cite{bh} for definitions).\\

Given two points in the boundary at infinity of a Hadamard space, the question of whether there exists a geodesic in the space which is forward asymptotic to one of the points and backward asymptotic to the other, is an example of an asymptotic Plateau problem. For higher dimensional varieties in the visual boundary, the asymptotic Plateau problem was first studied by Anderson in \cite{anderson_invent} for hyperbolic manifolds, where he showed that for any closed embedded manifold in the visual boundary, there exists an area minimizing local integral current which is asymptotic to the given manifold in the boundary. Extensions of Anderson's result for Hadamard manifolds of pinched negative curvature were obtained in \cite{gromov_foliated} and \cite{bangert_lang} and then for Gromov hyperbolic manifolds with bounded geometry in \cite{lang_cvpde}.\\

In the abovementioned results, the underlying space never contains a flat. In a space that contains a flat, the situation becomes more subtle and by considering Euclidean space, it is clear that the results above do not hold in such spaces; in this regard it is interesting to take a look at \cite[Corollary 5.2]{kloeckner_mazzeo} where Kloeckner and Mazzeo give a complete classification of the curves in the visual boundary of \(\mathbb{H}^2\times \R\) for which the asymptotic Plateau problem can be solved.\\

In this paper, we build upon work of Kleiner and Lang in \cite{kl_hrh} and consider the asymptotic Plateau problem for cycles in the Tits boundary of a proper Hadamard space. By a {\it cycle} is meant an integral metric current without boundary; given a complete metric space \(X\), the group of \(k\)-dimensional integral metric currents with finite mass in \(X\) is denoted by \(\Iak{k}(X)\) and in case \(X\) is locally compact, the group of \(k\)-dimensional integral metric currents in \(X\) with locally finite mass is denoted by \(\Iloc{k}(X)\). The respective subgroups of cycles, i.e. currents without boundary, are denoted by \(\Zak{k}(X)\) and \(\Zloc{k}(X)\). In Subsection \ref{current_sect}, we give a summary of what we need from the theory of currents. Before we state our main result, let us recall the following notions from \cite{kl_hrh}: Let \(S \in \Zloc{k}(X)\) be a local cycle. The {\it asymptotic density} of \(S\) is defined by
\[
\Theta_\infty(S) \coloneqq \limsup_{r\rightarrow \infty} \frac{1}{r^k}\|S\|(B(p,r))
\]
where \(p\in X\) is some point. For \(p\in X\) and \(r > 0\),
\[
F_{p,r}(Z)\coloneqq \frac{1}{r^{k+1}}\cdot \inf \{\Mass(V) \mid V\in \Iak{k+1}(X),\, \spt(Z - \partial V) \cap B(p,r) = \emptyset\}
\]
and the {\it asymptotic filling density} of \(S\) is defined by
\[
F_\infty(S) \coloneqq \limsup_{r\rightarrow \infty} F_{p,r}(S).
\]
Clearly, \(\Theta_\infty(S)\) and \(F_\infty(S)\) do not depend on the choice of \(p \in X\). If \(S' \in \Zloc{k}(X)\) is another local cycle and \(F_\infty(S-S') = 0\) then one says that \(S\) and \(S'\) are {\it \(F\)-asymptotic}.\\

Our main result, which is contained in Section \ref{solving}, can now be stated as follows:

\begin{theorem}\label{main_thm}
Let \(X\) be a proper Hadamard space, \(k\geq 1\) an integer and \(\bar{Z}\in \Zak{k-1}(\partial_TX)\) a strongly immovable cycle. For each \(p\in X\) we let \(R_p \in \Zloc{k}(X)\) denote the cone from \(p\) over \(\bar{Z}\). The following holds:
\begin{enumerate}[label=\emph{(\roman*)}]
    \item There exists an absolutely area minimizing local cycle \(S\in \Zloc{k}(\partial_TX)\) which satisfies that \(F_\infty(S-R_p) = 0\) for every \(p\in X\).
    \item For every \(p\in X\) and every minimizer \(S\in \Zloc{k}(X)\) satisfying \(F_\infty(S-R_p) = 0\) it holds that \((\varrho_{p,\lambda})_\#S \rightarrow R_p\) in the local flat topology as \(\lambda \searrow 0\). Here \(\varrho_{p,\lambda}\) denotes the \(\lambda\)-dilation of \(X\) with respect to \(p\), \(0 \leq \lambda \leq 1\).
    \item For every \(p\in X\) and every \(\varepsilon > 0\) there is an \(\bar{r} > 0\) such that if \(S\in \Zloc{k}(X)\) is a minimizer satisfying \(F_\infty(S-R_p) = 0\) and \(x \in \spt(S)\) with \(d(p,x) \geq \bar{r}\) then \(d(x,\spt(R_p)) < \varepsilon \cdot d(p,x)\).
    \item If \(S\in \Zloc{k}(X)\) is a minimizer with \(F_\infty(S-R_p) = 0\) then \(\partial_T\spt(S) = \spt(\bar{Z})\) and \(\Theta_\infty(S) = \frac{1}{k}\Mass(\bar{Z})\).
\end{enumerate}
\end{theorem}

Strongly immovable cycles are defined in Section \ref{imm_sect}; informally speaking, the concept is supposed to capture the idea that the cycle is isolated when viewed as an element of any asymptotic cone with a fixed base point. The notion is a stricter version of an immovability condition introduced by Huang in \cite{huang} and it is also inspired by the works of Huang, Kleiner and Stadler in \cite{hks_morse1,hks_morse2}. It is a classical result \cite[Proposition 9.21(1)]{bh} that given two points in the boundary at infinity of a Hadamard space, one can find a geodesic which is forward asymptotic to one of the points and backward asymptotic to the other, if one the points is isolated in the Tits boundary. Hence it is natural to try to solve the asymptotic Plateau problem for cycles in the Tits boundary which satisfy some isolation condition.\\

Let us now compare our result with the ones of Kleiner and Lang from \cite{kl_hrh}. In \cite{wenger_asymptotic}, Wenger introduced the notion of {\it asymptotic rank} for general metric spaces. It follows from work of Kleiner in \cite{kleiner_local} that for proper and cocompact Hadamard spaces, the asymptotic rank and the Euclidean rank agree, cf. \cite[Theorem 3.4]{wenger_asymptotic}. Now, let \(X\) be a proper Hadamard space of asymptotic rank \(m\geq 1\) (actually, Kleiner and Lang proved their results more generally for spaces with a convex geodesic bicombing). Kleiner and Lang showed that for a cycle \(\bar{Z}\in \Zak{m-1}(\partial_TX)\), the conclusions in Theorem \ref{main_thm} hold; cf. \cite[Theorem 5.6]{kl_hrh} for (i), \cite[Theorem 7.3]{kl_hrh} for (ii), \cite[Theorem 8.1(2)]{kl_hrh} for (iii), \cite[Proposition 8.2]{kl_hrh} for (iv) and the comment following \cite[Theorem 9.5]{kl_hrh} for (v). In fact, they show that for every cycle with mass less than or equal to some given number \(M\), the same \(\bar{r}\) in (iii) can be chosen, given \(\varepsilon > 0\) and \(p \in X\).

\subsection*{Acknowledgments} 

I want to thank Urs Lang for introducing the asymptotic Plateau problem to me, for many useful discussions, for help with formulating some parts of the paper and for reading the manuscript carefully, spotting some mistakes and providing valuable feedback. Stephan Stadler is thanked for useful discussions about the asymptotic Plateau problem, feedback on this work and for pointing out a mistake in a draft of this paper. Sigurður Jens Albertsson is thanked for reading through the manuscript, catching several typos and grammatical errors and providing suggestions on wording. I want to express my gratitude for financial support from the Swiss National Science Foundation, grant 197090.

\section{Preliminaries}

In this section we will fix notation and recall some definitions and results that we will need.

\subsection{Basics}

Given a metric space \((X,d)\), a point \(p\in X\) and \(r\geq 0\) we let \(U(p,r)\), \(B(p,r)\) and \(S(p,r)\) denote the open ball, the closed ball and the sphere, respectively, of radius \(r\) centered at \(p\). If it is not clear what the underlying space is, we will sometimes add \(X\) as subscript and write \(U_X(p,r)\), \(B_X(p,r)\) and \(S_X(p,r)\). Recall that a metric space is said to be {\it proper} if closed balls are compact. Given a point \(p\in X\), we let \(d_p(x)\coloneqq d(p,x)\).\\

Let \((X,d_X),\, (Y,d_Y)\) be metric spaces and \(f: X\rightarrow Y\) a map. The map \(f\) is said to be {\it \(L\)-Lipschitz} if \(L \geq 0\) and
\[
d_Y(f(x),f(y)) \leq L \cdot d_X(x,y)
\]
for all \(x,y\in X\). It is said to be {\it \(L\)-bi-Lipschitz} if \(L > 0\) and
\[
\frac{1}{L}\cdot d_X(x,y)\leq d_Y(f(x),f(y)) \leq L\cdot d_X(x,y)
\]
for all \(x,y \in X\). It is said that \(f\) is {\it Lipschitz (bi-Lipschitz)} if it is \(L\)-Lipschitz (\(L\)-bi-Lipschitz) for some \(L\). If \(f\) is Lipschitz, then we let \(\Lip(f)\) denote the smallest number \(L\) such that \(f\) is \(L\)-Lipschitz.\\

A {\it Hadamard space} is a complete geodesic metric space \((X,d)\) which satisfies the following comparison principle: Let \(o,x,y \in X\), \(\alpha,\beta: [0,1]\rightarrow X\) be geodesics such that \(\alpha(0) = \beta(0) = o\), \(\alpha(1) = x\), \(\beta(1) = y\). Let \(\bar{x},\bar{y}\in \R^2\) be such that \(\|\bar{x}\| = d(o,x)\), \(\|\bar{y}\| = d(o,y)\) and \(\|\bar{y}-\bar{x}\| = d(x,y)\). Then
\[
d(\alpha(s),\beta(t)) \leq \|t\bar{y}-s\bar{x}\|
\]
for every \(s,t\in[0,1]\). When \(X\) is a Hadamard space, we let \(\sigma: [0,1]\times X \times X \rightarrow X\) denote the geodesic bicombing, i.e. the map such that for every \(x,y\in X\), \(t\mapsto \sigma(t,x,y)\) is the unique geodesic with \(\sigma(0,x,y) = x\) and \(\sigma(1,x,y) = y\). For \(p\in X\) and \(\lambda \in [0,1]\) we let \(\varrho_{p,\lambda}(x)\coloneqq \sigma(\lambda,p,x)\) denote the {\it \(\lambda\)-dilation with respect to \(p\)}. The map \(\varrho_{p,\lambda}\) is \(\lambda\)-Lipschitz.

\subsection{Currents in metric spaces}\label{current_sect}

In \cite{ambrosio_kirchheim}, Ambrosio and Kirchheim introduced the notion of currents with finite mass for arbitrary complete metric spaces. Later, in \cite{lang_local}, Lang showed that in the case of locally compact metric spaces, the finite mass condition from Ambrosio and Kirchheim can be dispensed with, and hence that for these spaces there is a theory of local currents, similar to the classical theory of currents in \(\R^n\) developed by Federer and Fleming in \cite{federer_fleming}. We will now give a brief summary of what we need from the theory of currents in metric spaces. References are \cite{ambrosio_kirchheim} and \cite{lang_local}. We also used \cite[Section 2.2]{wenger_euclidean} and \cite[Section 2]{kl_hrh}.

\subsubsection{Definition of a current}

From now on, \(X\) will denote a complete metric space. In case we are talking about local currents in \(X\) then we will further assume that \(X\) is locally compact.\\

Let \(k\geq 0\) be an integer. A \(k\)-dimensional {\it metric current} with (locally) finite mass is an \(\R\) valued \((k+1)\)-linear function \(T\) defined on \((k+1)\)-tuples \((f,\pi_1,\ldots,\pi_k)\) of (locally) Lipschitz functions on \(X\) with \(f\) bounded (compactly supported) which satisfies the following conditions:
\begin{enumerate}[label=(\roman*)]
    \item \(T(f,\pi_{1,j},\ldots,\pi_{k,j}) \rightarrow T(f,\pi_1,\ldots,\pi_k)\) as \(j\rightarrow \infty\) whenever \(\pi_{i,j}\rightarrow \pi_i\) pointwise as \(j\rightarrow \infty\) and \(\sup_{j\in \N}\Lip(\pi_{i,j}) < \infty\) (\(\sup_{j\in \N} \Lip(\pi_{i,j}|_K) < \infty\) for every compact set \(K\subseteq X\)) for every \(i\in \{1,\ldots,k\}\).
    \item \(T(f,\pi_1,\ldots,\pi_k) = 0\) if there exists \(i \in \{1,\ldots,k\}\) such that \(\pi_i\) is constant on a neighborhood of \(\spt(f)\).
    \item There exists a (locally) finite Borel measure \(\mu\) on \(X\) such that 
    \[
    |T(f,\pi_1,\ldots,\pi_k)| \leq \prod_{i=1}^k \Lip(\pi_i) \cdot \int_X |f|\: d\mu
    \]
    for every \((f,\pi_1,\ldots,\pi_k)\).
\end{enumerate}
The smallest measure which satisfies property (iii) is called the {\it mass measure} of \(T\) and denoted by \(\|T\|\). The {\it total mass} of \(T\) is \(\Mass(T)\coloneqq \|T\|(X)\). The space of \(k\)-dimensional metric currents in \(X\) with (locally) finite mass is denoted by \(\Mak{k}(X)\) (\(\Mloc{k}(X)\)).

\subsubsection{Restrictions, support and push-forward}

A \(k\)-dimensional current \(T\) in \(X\) with (locally) finite mass extends uniquely to the space of \((k+1)\)-tuples \((f,\pi_1,\ldots,\pi_k)\) with \(f\) a bounded (and compactly supported) Borel function and \(\pi_1,\ldots,\pi_k\) (locally) Lipschitz functions. This allows one to define the {\it restriction} of \(T\) to a Borel set \(B\subseteq X\) by 
\[
(T\res B)(f,\pi_1,\ldots,\pi_k) \coloneqq T(f\cdot \mathbf{1}_B, \pi_1,\ldots,\pi_k)
\]
where \(\mathbf{1}_B\) denotes the characteristic function of \(B\). The {\it support} of \(T\) is defined as 
\[
\spt(T) \coloneqq \spt(\|T\|) = \{x \in X \mid \|T\|(B(x,\varepsilon)) > 0 \text{ for every } \varepsilon > 0\}.
\]
Let \(\varphi: X\rightarrow Y\) be a (proper (i.e. the preimage of every compact set is compact) locally) Lipschitz map and \(T\) a current in \(X\). Then the {\it push forward} \(\varphi_\# T\in \Mak{k}(Y)\) (\(\varphi_\#T\in \Mloc{k}(Y)\)) of \(T\) under \(\varphi\) is defined by
\[
(\varphi_\#T)(f,\pi_1,\ldots,\pi_k) = T(f\circ \varphi, \pi_1\circ \varphi,\ldots,\pi_k\circ \varphi).
\]
It holds that
\[
\Mass(\varphi_\# T) \leq \Lip(\varphi)^k\cdot \Mass(T).
\]

\subsubsection{The boundary operator, normal currents, rectifiable currents, integer rectifiable currents and integral currents}

If \(k\geq 1\), then the {\it boundary} of \(T\) is the function \(\partial T\) defined on \(k\)-tuples \((f,\pi_1,\ldots,\pi_{k-1})\) of (locally) Lipschitz functions with \(f\) bounded (compactly supported) via 
\[
(\partial T)(f,\pi_1,\ldots,\pi_{k-1}) \coloneqq T(\mathbf{1}_{\spt(T)\cap \spt(f)},f,\pi_1,\ldots,\pi_{k-1}).
\]
If \(\partial T\) is a current with (locally) finite mass or \(k=0\) then \(T\) is called a ({\it locally}) {\it normal current}; the space of those currents is denoted by \(\Nak{k}(X)\) (\(\Nloc{k}(X)\)).

\begin{example}\cite[p. 3]{ambrosio_kirchheim}\cite[Proposition 2.6]{lang_local}
Let \(X = \R^k\) and \(\theta: \R^k \rightarrow \R\) be a locally integrable function. Then the map \(\bb{\theta}\) which is defined on tuples \((f,\pi_1,\ldots,\pi_k)\) of compactly supported Lipschitz functions as follows
\[
\bb{\theta}(f,\pi_1,\ldots,\pi_k) \coloneqq \int_{\R^n}\theta(x) f(x) \det\left(\frac{\partial \pi_i}{\partial x_j}(x)\right)\: dx
\]
is a \(k\)-dimensional current in \(\R^k\). Its mass measure is given by
\[
\|T\|(A) = \int_A |\theta(x)| \: dx.
\]
\end{example}

The \(k\)-dimensional metric currents \(T\) with (locally) finite mass for which there exists a sequence \(K_j\subseteq \R^k\) of compact sets together with integrable functions \(\theta_j: K_j \rightarrow \R\) and bi-Lipschitz maps \(\varphi_j: K_j \rightarrow X\) such that 
\[
T = \sum_{j=1}^\infty \varphi_{j\#}\bb{\theta_j} \quad \text{and} \quad \|T\| = \sum_{j=1}^\infty \|\varphi_{j\#}\bb{\theta_j}\|
\]
are said to be ({\it locally}) {\it rectifiable} and the space of (locally) rectifiable metric currents in \(X\) is denoted by \(\sRak{k}(X)\) (\(\sRloc{k}(X)\)). Note that if \(k = 0\), this means that there exist sequences \(x_j\in X\) and \(\theta_j \in \R\) such that \(T(f) = \sum_{j=1}^\infty \theta_j\cdot f(x_j)\) for every bounded Lipschitz function \(f\). If the numbers \(\theta_j\) can be taken to be integer valued, then \(T\) is called (a {\it locally}) an {\it integer rectifiable current} and the corresponding space is denoted by \(\sIak{k}(X)\) (\(\sIloc{k}(X)\)).\\

By the so called {\it boundary rectifiability theorem} it holds that if \(T\) is integer rectifiable and \(\partial T\) has (locally) finite mass, then \(\partial T\) is also integer rectifiable. (A locally) An integer rectifiable current which is normal is called (a {\it locally}) an {\it integral current} and the corresponding space is denoted by \(\Iak{k}(X)\) (\(\Iloc{k}(X)\)). It is a fact that \(\partial \circ \partial = 0\) so the spaces \(\Iak{k}(X)\) (\(\Iloc{k}(X)\)), \(k\geq 0\), together with the maps \(\partial\) form a chain complex. For \(k\geq 1\) we let \(\Zak{k}(X)\) (\(\Zloc{k}(X)\)) denote the elements of \(\Iak{k}(X)\) (\(\Iloc{k}(X)\)) with zero boundary. For \(k=0\), \(\Zak{k}(X)\) (\(\Zloc{k}(X)\)) denotes the set of currents \(T \in \Iak{k}(X)\) (\(T \in \Zloc{k}(X)\)) which are given by \(T(f) = \sum_{j=1}^n \theta_j \cdot f(x_j)\) where \(n\geq 1\) is an integer, \(x_j\) points in \(X\) and \(\theta_j\) integers such that \(\sum_{j=1}^n\theta_j = 0\). The elements of \(\Zak{k}(X)\) (\(\Zloc{k}(X)\)) are referred to as ({\it local}) {\it integral cycles.}

\subsubsection{Intrinsic representation of rectifiable currents}

Let \(k\geq 0\) be an integer and \(T \in \Mak{k}(X)\). The {\it set} of \(T\) is defined as follows:
\[
\set(T) \coloneqq \left\{x \in X \;\Big|\; \liminf_{r \searrow 0} r^{-k}\cdot\|T\|(B(x,r)) > 0\right\}.
\]
It is shown in \cite{ambrosio_kirchheim} that if \(T\) is rectifiable then \(\sH^k(\spt(T)\ssm \set(T)) = 0\).\\

We will need the following theorem from \cite{ambrosio_kirchheim} which is stated there for so called \(w^*\)-separable dual spaces, a class of spaces which in particular contains \(\ell^\infty(\N)\). Under the assumption – which we make – that the cardinality of any set is an Ulam number, a condition which is consistent with standard ZFC set theory, it holds that any finite Borel measure is concentrated on its support. As the support of a finite measure is separable, it can be embedded isometrically into \(\ell^\infty(\N)\) and hence the following formulation of the theorem is justified (cf. discussion following Definition 2.8 in \cite{ambrosio_kirchheim}, discussion on page 4 and Lemma 2.9 in \cite{ambrosio_kirchheim}).

\begin{theorem}\label{intr_rep}\emph{\cite[Theorem 9.1]{ambrosio_kirchheim}}
Let \(k\geq 0\) be an integer and \(T\in \sRak{k}(X)\). Then there exists a Borel function \(\theta: \set(T)\rightarrow (0,\infty)\) with \(\int_{\set(T)}\theta \: d\sH^k < \infty\) and a measurable orientation \(\tau\) of \(\set(T)\) such that
\[
T(f,\pi_1,\ldots,\pi_k) = \int_{\set(T)} f(x)\theta(x) \langle \Alt_k D\pi(x)|_{T_x\set(T)},\tau(x)\rangle \: d\sH^k(x)
\]
for every \(f,\pi_1,\ldots,\pi_k\) Lipschitz, with \(f\) bounded. Here \(T_x\set(T)\) denotes the approximate tangent space to \(\set(T)\) at \(x\) and
\[
\Alt_k D\pi(x)|_{T_x\set(T)} = D\pi_1(x)|_{T_x\set(T)} \wedge \cdots \wedge D\pi_k(x)|_{T_x\set(T)}
\]
where \(Dg(x)|_{T_x\set(T)}\) denotes the tangential derivative of \(g\) at \(x\) when \(g\) is Lipschitz function on \(\set(T)\) (cf. discussion in section 9.1 in \cite{ambrosio_kirchheim}, preceding Theorem 9.1).
\end{theorem}

We also need the following theorem.

\begin{theorem}\emph{\cite[Theorem 9.5]{ambrosio_kirchheim}}
Let \(k\geq 0\) be an integer and \(T\in \sRak{k}(X)\). Then \(\|T\| = \theta \lambda \sH^k\res \set(T)\) i.e.
\[
\|T\|(B) = \int_{B\cap \set(T)} \theta(x)\lambda(x)\: d\sH^k
\]
for every Borel set \(B\), where \(\theta\) is the function from Theorem \ref{intr_rep} and \(\lambda(x)\coloneqq \lambda_T(x)\) is the area factor of \(T_x\set(T)\) (cf. discussion around (9.11) in \cite{ambrosio_kirchheim}).
\end{theorem}

It will be used in the proof of Lemma \ref{cancellation} that for every \(T \in \sRak{k}(X)\), all Lipschitz functions \(\pi_1,\ldots,\pi_k\) on \(X\) and every measurable orientation \(\tau\) of \(\set(T)\), that
\begin{equation}\label{area_factor_ineq}
\left|\langle \Alt_k D\pi(x)|_{T_x\set(T)},\tau(x)\rangle\right| \leq \lambda_T(x)\cdot \prod_{i=1}^k \Lip(\pi_i)
\end{equation}
where \(\lambda_T\) is the area factor of \(\set(T)\) (cf. beginning of the proof of \cite[Theorem 9.5]{ambrosio_kirchheim}).

\subsubsection{Slicing}

Let now \(\pi: X \rightarrow \R\) be a Lipschitz map, \(k\geq 1\) an integer and \(T\) a \(k\)-dimensional (locally) normal current. Then the {\it slice} of \(T\) with respect to \(\pi\) at \(r \in \R\) is defined by
\[
\langle T,\pi,r\rangle \coloneqq \partial (T\res \{\pi \leq r\}) - (\partial T)\res \{\pi \leq r\}.
\]
The following is referred to as the {\it slicing theorem}: For almost every \(r\), the slice \(\langle T,\pi, r\rangle\) is a \((k-1)\)-dimensional (locally) normal current whose support is contained in \(\spt(T)\cap \{\pi = r\}\) and if \(T\) is a (locally) integral current, then \(\langle T,\pi,r\rangle\) is a (locally) integral current for almost every \(r\). Furthermore, for every \(a < b\) the following coarea inequality holds:
\[
\int_a^b \Mass(\langle T,\pi,r\rangle) \: dr \leq \Lip(\pi) \cdot \|T\|(\{a < \pi < b\}).
\]

\subsubsection{Convergence}

Let \(T_j,T\in \Iak{k}(X)\) (\(\Iloc{k}(X)\)), \(j\in \N\). One says that \(T_j\) {\it converges weakly} to \(T\) as \(j\rightarrow \infty\), written \(T_j \rightarrow T\) weakly as \(j\rightarrow \infty\), if for every \((f,\pi_1,\ldots,\pi_k)\) it holds that \(T_j(f,\pi_1,\ldots,\pi_k)\rightarrow T(f,\pi_1,\ldots,\pi_k)\) as \(j\rightarrow \infty\). Weak convergence commutes with the boundary operator i.e. \(\partial T_j \rightarrow \partial T\) weakly if \(T_j\rightarrow T\) weakly. Mass is lower semi-continuous with respect to weak convergence, i.e. \(\Mass(T) \leq \liminf_{j\rightarrow \infty}\Mass(T_j)\) if \(T_j\rightarrow T\) weakly and more generally, \(\|T\|(U)\leq \liminf_{j\rightarrow \infty} \|T_j\|(U)\) for every open set \(U\subseteq X\).

\subsubsection{Fillings and minimizers}

Let \(k\geq 1\) be an integer and \(Z \in \Zak{k-1}(X)\) be a cycle. A chain \(V\in \Iak{k}(X)\) is called a {\it filling} of \(Z\) if \(\partial V = Z\). The {\it filling volume} of \(Z\) is defined by
\[
\FillVol(Z)\coloneqq \inf \{\Mass(V) \mid V \in \Iak{k}(X),\, \partial V = Z\}.
\]
A filling \(V\in \Iak{k}(X)\) of \(Z\) is said to be {\it minimal} if \(\Mass(V) = \FillVol(Z)\). If \(X\) is locally compact and \(Z\) admits a filling, then it admits a minimal filling by the compactness theorem of Lang \cite[Theorem 8.10]{lang_local}.\\

A chain \(V \in \Iak{k}(X)\) is said to be {\it absolutely area minimizing} (or just {\it minimizing} for short) if \(V\) is a minimal filling of \(\partial V\). A local chain \(V\in \Iloc{k}(X)\) is said to be {\it absolutely area minimizing} if for every Borel set \(B\subseteq X\) such that \(V\res B \in \Iak{k}(X)\), \(V\res B\) is minimizing.

\subsubsection{Flat distance and flat convergence}

Given \(S,T \in \Iak{k}(X)\) the {\it flat distance} \(d_\sF(S,T)\) between them is defined by
\[
d_\sF(S,T) \coloneqq \inf\{\Mass(T-S - \partial V) + \Mass(V) \mid V \in \Iak{k+1}(X)\}.
\]
If \(T_j,T\in \Iak{k}(X)\), \(j\in \N\), and \(d_\sF(T,T_j)\rightarrow 0\) as \(j\rightarrow \infty\), one says that \(T_j\) converges to \(T\) in the flat topology as \(j\rightarrow \infty\). If \(T_j\) converges to \(T\) in the flat topology, it converges weakly to \(T\). It is a deep result due to Stefan Wenger that if \(X\) satisfies local cone-type inequalities, \(\sup_{j\in \N} (\Mass(T_j) + \Mass(\partial T_j)) < \infty\) and \(T_j \rightarrow T\) weakly, then \(T_j \rightarrow T\) in the flat topology as \(j\rightarrow \infty\). Further, if \(\partial T_j = 0\) for every \(j\in \N\) then \(\FillVol(T_j - T) \rightarrow 0\) as \(j\rightarrow \infty\) if \(T_j \rightarrow T\) weakly and \(\sup_{j\in \N} \Mass(T_j) < \infty\) \cite[Theorem 1.2]{wenger_flat}. A space \(X\) is said to satisfy {\it local cone-type inequalities} for \(\Iak{k}(X)\) if there exists a \(\delta > 0\) and a constant \(c_k \geq 0\) such that if \(Z\in \Zak{k}(X)\) and \(\diam(\spt(Z)) \leq \delta\) then there exists a filling \(V\in \Iak{k+1}(X)\) of \(Z\) with \(\Mass(V) \leq c_k\cdot \diam(\spt(Z))\cdot \Mass(Z)\). A space \(X\) is said to satisfy {\it cone-type inequalities} for \(\Iak{k}(X)\) if there exists a constant \(c_k > 0\) such that for every cycle \(Z\in \Zak{k}(X)\) there exists a filling \(V\in \Iak{k+1}(X)\) with \(\Mass(V)\leq c_k\cdot \diam(\spt(Z))\cdot \Mass(Z)\). A space which satisfies cone-type inequalities obviously satisfies local cone-type inequalities. By Theorem \ref{cone_ineq}, Hadamard spaces satisfy cone-type inequalities. Banach spaces also satisfy cone-type inequalities, cf. Proposition 2.10 and Section 2.1 in \cite{wenger_euclidean}.\\

Given \(T_j,T \in \Iloc{k}(X)\), \(j\in \N\), one says that \(T_j \rightarrow T\) in the {\it local flat topology} as \(j\rightarrow \infty\) if for every compact set \(K\subseteq X\) there is a sequence \(V_j \in \Iak{k+1}(X)\) such that 
\[
(\|T_j - T - \partial V_j\| + \|V_j\|)(K) \rightarrow 0
\]
as \(j\rightarrow \infty\). It follows from Wenger's result mentioned above that if 
\[
\sup_{j\in \N} (\|T_j\| + \|\partial T_j\|)(K) < \infty
\]
for every compact set \(K\subseteq X\) and \(T_j \rightarrow T\) weakly as \(j\rightarrow \infty\), then \(T_j\rightarrow T\) in the local flat topology as \(j\rightarrow \infty\).\\

An important property of (local) flat convergence is that the (local) flat limit of minimizers is a minimizer.

\subsection{Isoperimetric inequalities}

Now we will recall the cone-construction due to Ambrosio and Kirchheim and then some results on isoperimetric inequalities in Hadamard spaces, all of which are due to Stefan Wenger. The discussion is similar to \cite[Subsection 2.7]{kl_hrh}.

\subsubsection{Homotopies, cones, the optimal cone inequality and the monotonicity formula}

For a normal current \(T\in \Nak{k}(X)\) with bounded support, there is a product construction which gives a normal current \(\bb{0,1}\times T \in \Nak{k+1}([0,1]\times X)\) which satisfies
\[
\partial (\bb{0,1}\times T) = \bb{1}\times T - \bb{0}\times T - \bb{0,1}\times T.
\]
Here \(\bb{t}\times T \in \Nak{k}([0,1]\times X)\) is given by \((\bb{t}\times T)(f,\pi_1,\ldots,\pi_k) = T(f_t,(\pi_1)_t,\ldots,(\pi_k)_t)\) where \(f_t\coloneqq f(t,\cdot)\) and \((\pi_i)_t \coloneqq \pi_i(t,\cdot)\). If \(T\) is an integral current then \(\bb{0,1}\times T\) is also an integral current (cf. \cite[Subsection 2.3]{wenger_euclidean}). If \(h: [0,1]\times X \rightarrow X\) is Lipschitz, then 
\[
h_\#(\bb{t}\times T) = h_{t\#}T
\]
where \(h_t\coloneqq h(t,\cdot)\) so
\[
\partial h_\#(\bb{t}\times T) = h_{1\#}T - h_{0\#}T - h_\#(\bb{0,1}\times \partial T).
\]
Now, assume that \(X\) is a Hadamard space, fix \(p\in X\) and let \(\sigma_p: [0,1]\times X \rightarrow X\), \(\sigma_p(t,x)\coloneqq \sigma(t,p,x)\) (recall that \(\sigma\) is the geodesic bicombing of \(X\)). Then for any \(Z \in \Zak{k}(X)\) with bounded support it holds that
\[
\partial \sigma_{p\#}(\bb{0,1}\times Z) = (\sigma_p)_{1\#}Z - (\sigma_p)_{0\#}Z - \sigma_{p\#}(\bb{0,1}\times \partial Z) = Z,
\]
i.e. \(\sigma_{p\#}(\bb{0,1}\times Z)\) is a filling of \(Z\) and it is called the {\it cone over} \(Z\) from \(p\). Note in particular that any cycle in a Hadamard space has a filling. The following optimal cone-inequality is due to Stefan Wenger:

\begin{theorem}\label{cone_ineq}\emph{\cite[Theorem 4.1]{wenger_filling}}
Let \(X\) be a Hadamard space, \(Z \in \Zak{k}(X)\) and assume that \(\spt(Z) \subseteq B(p,r)\). Then 
\[
\Mass(\sigma_{p\#}(\bb{0,1}\times Z)) \leq \frac{r}{k+1}\cdot \Mass(Z).
\]
\end{theorem}

An important consequence of Theorem \ref{cone_ineq} is the following monotonicity formula.

\begin{theorem}\emph{\cite[Corollary 4.4]{wenger_filling}}
Let \(X\) be a Hadamard space and \(T\in \Iak{k}(X)\) (or \(T\in \Iloc{k}(X)\)) be minimizing. Then for every \(x\in \spt(T)\) and \(0 < r \leq s < d(p,\spt(\partial T))\) it holds that
\[
\frac{\|T\|(B(p,r))}{r^k} \leq \frac{\|T\|(B(p,s))}{s^k}.
\]
\end{theorem}

\subsubsection{The Euclidean isoperimetric inequality}

We will need Wenger's Euclidean isoperimetric inequality:

\begin{theorem}\emph{\cite[Theorem 1.2]{wenger_euclidean}}
Let \(X\) be a Hadamard space. For every integer \(k\geq 0\) there exists a number \(\gamma_k \geq 0\) such that for every \(Z\in \Zak{k}(X)\) there exists a filling \(V\in \Iak{k+1}(X)\) of \(Z\) with \(\Mass(V) \leq \gamma_k\cdot \Mass(Z)^\frac{k+1}{k}\).
\end{theorem}

The following is a well known fact which can be derived from the Euclidean isoperimetric inequality.

\begin{corollary}\label{lowerbound}
Let \(X\) be a Hadamard space. For every \(k\geq 1\) there exists a constant \(\theta_k > 0\) such that if \(T\in \Iak{k}(X)\) is minimizing then 
\[
\|T\|(B(x,r)) \geq \theta_k\cdot r^k
\]
for every \(0 \leq r \leq d(x,\spt(\partial T))\).
\end{corollary}

\subsection{Asymptotic cones}

The following definitions and facts are all well known, cf. e.g. \cite[pp. 77-81]{bh} and \cite[Section 2.4]{kleiner_leeb}. 

\begin{definition}
Let \(\sP(\N)\) denote the power set of \(\N\), i.e. the set of all subsets of \(\N\). A finitely additive function \(\omega: \sP(\N) \rightarrow \{0,1\}\) which satisfies that \(\omega(\N) = 1\) and \(\omega(S) = 0\) for every finite \(S\subseteq \N\) is called a {\it non-principal ultrafilter} on \(\N\).
\end{definition}

Let \((x_j)_{j\in \N}\) be a sequence of points in a metric space \(X\). One says that \(x = \lim_\omega x_j\) if for every \(\varepsilon > 0\) it holds that \(\omega(\{j\in \N \mid d(x,x_j) < \varepsilon\}) = 1\). If all the points \(x_j\), \(j\in \N\), are contained in a compact set then there exists a unique \(x\in X\) such that \(x = \lim_\omega x_j\).

\begin{definition}
Let \((X_j,d_j,p_j)\), \(j\in \N\), be a sequence of pointed metric spaces and \(\omega\) be a non-principal ultrafilter on \(\N\) and for each \(j\in \N\), let \(Y_j\subseteq X_j\). The {\it ultralimit} \(\lim_\omega (Y_j,d_j,p_j)\) is the metric space which one obtains from the space
\[
\bigg\{(x_j)_{j\in \N} \in \prod_{j\in \N} Y_j\;\Big|\; \sup_{j\in \N}d_j(p_j,x_j) < \infty\bigg\},
\]
endowed with pseudometric 
\[
d_\omega((x_j)_{j\in\N},(y_j)_{j\in\N}) \coloneqq \lim_\omega d_j(x_j,y_j),
\]
after identifying points at zero distance; the equivalence class of \((x_j)_{j\in \N}\) is denoted by \([(x_j)_{j\in\N}]\). The metric on \(\lim_\omega (Y_j,d_j,p_j)\) is still denoted by \(d_\omega\).
\end{definition}

Ultralimits are always complete metric spaces.

\begin{definition}
Let \((X,d)\) be a metric space, \(0 < r_j \nearrow \infty\) and \(p_j \in X\), \(j\in \N\), be sequences and \(\omega\) a non-principal ultrafilter on \(\N\). The ultralimit \(\lim_\omega (X,\frac{1}{r_j}d,p_j)\) is called an {\it asymptotic cone of} \(X\).
\end{definition}

If \((X,d)\) is a metric space, \(p\in X\), \(0 < r_j\nearrow \infty\), \(j\in \N\), a sequence and \(\omega\) a non-principal ultrafilter on \(\N\), then it is customary to let \(p_\omega \coloneqq [(p)_{j\in \N}]\in \lim_\omega (X,\frac{1}{r_j}d,p)\).\\

We need the following fact.

\begin{proposition}\emph{\cite[Proposition 2.2]{wenger_compact}}
\label{embedding_lemma}
Let \((X_j,d_j,p_j)\), \(j\in \N\), be a sequence of pointed complete metric spaces, \(k\geq 0\) an integer and for each \(j\in \N\), let \(T_j \in \Iak{k}(X_j)\). Assume that \(\sup_{j\in \N} d_j(p_j,\spt(T_j)) < \infty\), \(\sup_{j\in \N}\diam(\spt(T_j)) < \infty\) and that there exists a complete metric space \(Z\), isometric embeddings \(\varphi_j: X_j \rightarrow Z\) and \(T\in \Iak{k}(Z)\) such that \(\varphi_{j\#}T_j \rightarrow T\) weakly as \(j\rightarrow \infty\). Then for every non-principal ultrafilter \(\omega\) on \(\N\) there exists an isometric embedding \(\varphi: \spt(T) \rightarrow \lim_\omega (\spt(T_j),d_j,p_j)\).
\end{proposition}

\subsection{Tits cone and Tits boundary}

Most of the following subsection is an adaption of the discussion on Tits cones and Tits boundaries in \cite{kl_hrh}. Also, cf. \cite{bh}.\\

Let \(X\) be a Hadamard space. On the set
\[
\R_+X \coloneqq \{ \xi: \R_+ \rightarrow X\mid \xi \text{ is a geodesic}\}
\]
of rays in \(X\) of arbitrary speed we define the equivalence relation \(\sim\) by 
\[
\xi \sim \eta \quad \Longleftrightarrow \quad \sup_{t\geq 0}d(\xi(t),\eta(t)) < \infty.
\]
The equivalence class of \(\xi \in \R_+X\) is denoted by \([\xi]\). The set of equivalence classes \(\R_+X/\sim\) endowed with the Tits metric 
\[
d_T([\xi],[\eta]) \coloneqq \lim_{t\rightarrow \infty}\frac{1}{t}d(\xi(t),\eta(t)), \quad \xi,\eta \in \R_+X,
\]
forms a Hadamard space which is called the {\it Tits cone} of \(X\) and denoted by \(C_TX\). The constant rays form an equivalence class which is denoted by \(o\). The unit sphere centered at \(o\) is denoted by \(\partial_T X\) and called the {\it Tits boundary} of \(X\). The Tits cone is the Euclidean cone over \(\partial_TX\) endowed with the angular metric \(\angle_T\) defined by
\[
\angle_T (u,v) \coloneqq \sup_{p\in X} \angle_p (u,v)
\]
for every \(u,v\in \partial_TX\). Here \(\angle_p(u,v)\) denotes the angle at \(p\) between the unit speed geodesics emanating from \(p\) towards \(u\) and \(v\), respectively.\\

For each \(\lambda \in \R\) there is a homothety 
\[
C_TX \longrightarrow C_TX, \qquad [\xi]\longmapsto \lambda [\xi]\coloneqq [\xi(\lambda\,\cdot)].
\]

For every \(p\in X\) and every \(v\in C_TX\) there exists a unique ray \(\xi \in v\) which emanates from \(p\) so \(C_TX\) can be identified with the set of rays emanating from \(p\). We let
\[
\can_p: C_TX \longrightarrow X, \qquad v \longmapsto \xi(1),
\]
where \(\xi \in v\) is the unique ray emanating from \(p\). The map \(\can_p\) is \(1\)-Lipschitz.\\

The {\it Tits boundary of a set} \(A\subseteq X\) is denoted by \(\partial_TA\) and defined as the set of points \(u\in \partial_TX\) such that \(\frac{1}{r_j}d(x_j,\can_p(r_ju))\rightarrow 0\) as \(j\rightarrow \infty\) for some sequences \(x_j \in A\) and \(0 < r_j \nearrow \infty\), \(j\in \N\).\\

We will need the following lemma due to Bernhard Leeb:

\begin{lemma}\label{asympttits}\emph{\cite[Lemma 10.6]{kleiner_local}}
Let \((X,d)\) be a proper Hadamard space, \(p\in X\) a point, \(0 < r_j \nearrow \infty\), \(j\in \N\), a sequence, \(\omega\) a non-principal ultrafilter on \(\N\) and \(X_\omega \coloneqq \lim_\omega (X, \frac{1}{r_j}d,p)\). Then the map
\[
C_TX \longrightarrow X_\omega,\qquad v \longmapsto [(\can_p(r_jv))_{j\in\N}],
\]
is an isometric embedding.
\end{lemma}

\subsection{Conical cycles}

Let \(X\) be a proper Hadamard space and \(p\in X\) a point. A cycle \(R\in \Zloc{k}(X)\) is said to be {\it conical with respect to} \(p\) if \((\varrho_{p,\lambda})_\# R = R\) for every \(0 < \lambda \leq 1\). Given a cycle \(\bar{Z} \in \Zak{k-1}(\partial_TX)\), let \(\bar{R}\coloneqq h_\#(\bb{0,1}\times \bar{Z}) \in \Iak{k}(C_TX)\), where 
\[
h: \R \times C_TX \longrightarrow C_TX, \qquad (\lambda,v)\longmapsto \lambda v,
\]
denote the cone over \(\bar{Z}\) from \(o\in C_TX\). Then, for every \(0 \leq r \leq 1\) it holds that \(h_{r\#}\bar{R} = \bar{R}\res B(0,r)\), where \(h_r(u)\coloneqq r u\), \(u\in C_TX\), so there is a well defined cycle \(R\in \Zloc{k}(X)\) which is conical with respect to \(p\in X\) and determined by the condition that \(\pi_{r\#}\bar{R} = R\res B(p,r)\) for every \(r \geq 0\). Here \(\pi_r(u) \coloneqq \can_p(ru)\), \(u\in C_TX\). Furthermore, \(\Theta_\infty(R) \leq \Mass(\bar{R})\). The cycle \(R\) is called the {\it the cone from \(p\) over \(\bar{Z}\).} A partial converse is given by the following theorem due to Kleiner and Lang.

\begin{theorem}\label{lift}\emph{\cite[Theorem 9.3]{kl_hrh}}
Let \(X\) be a proper Hadamard space and \(p\in X\) a point. Assume that \(R\in \Zloc{k}(X)\) is conical with respect to \(p\), \(\Theta\coloneqq \Theta_\infty(R) < \infty\) and that \(\partial_T \spt(R)\) is a compact subset of \(\partial_TX\). Then there exists \(\bar{R}\in \Iak{k}(C_TX)\) such that \(h_{r\#}\bar{R} = \bar{R}\res B(o,r)\) for every \(0 \leq r \leq 1\) and \(\pi_{r\#}\bar{R} = R\res B(p,r)\) for every \(r > 0\). Furthermore, it holds that \(\|\bar{R}\|(B(o,r)) = \Theta \cdot r^k\) for every \(0 \leq r \leq 1\).
\end{theorem}

\begin{remark}
We refer to \(\bar{R}\) as the {\it lift} of \(R\) to \(C_TX\).
\end{remark}

\section{Immovable and strongly immovable cycles}\label{imm_sect}

The notion of immovable cycles, a sort of a higher dimensional analogue of isolated points, was originally introduced by Jingyin Huang for singular homology classes in \cite{huang}. We want to formulate the concept for integral cycles and for that we propose the following definition.

\begin{definition}(cf. \cite[Definition 1.3]{huang}) 
Let \(X\) be a complete metric space, \(k\geq 0\) and integer. A cycle \(Z\in \Zak{k}(X)\) is said to be {\it immovable} if for every chain \(T\in \Iak{k+1}(X)\) it holds that
\[
\Mass(Z + \partial T) = \Mass(Z) + \Mass(\partial T).
\]
\end{definition}

Now we introduce the class of cycles for which we will solve the asymptotic Plateau problem in Section \ref{solving}. From now on in this section, \(X\) will denote a fixed Hadamard space. 

\begin{definition}\label{str_imm_def}
Let \(k\geq 1\) be an integer, \(\bar{Z}\in \Zak{k-1}(\partial_T X)\) and \(\bar{R}\in \Iak{k}(C_T X)\) the cone from \(o\) over \(\bar{Z}\). One says that \(\bar{Z}\) is {\it strongly immovable} or {\it immovable in the strong sense} if for every sequence \(0 < r_j \nearrow \infty\), some point \(p\in X\) and every non-principal ultrafilter \(\omega\) on \(\N\) it holds that
\begin{equation}\label{imm_cond}
\Mass(\bar{T} - \iota_\#\bar{R}) + \Mass(\bar{R}) = \Mass(\bar{T})
\end{equation}
for every \(\bar{T} \in \Iak{k}(X_\omega)\) with \(\partial \bar{T} = \iota_\#\bar{Z}\), i.e. every filling of \(\iota_\#\bar{Z}\) contains \(\iota_\#\bar{R}\) as a piece. Here \(X_\omega \coloneqq \lim_\omega (X, \frac{1}{r_j}d,p)\) and \(\iota: C_TX \rightarrow X_\omega\) denotes the canonical embedding from Lemma \ref{asympttits}.
\end{definition}

\begin{remark}
As \(\lim_\omega (X,\frac{1}{r_j}d,p)\) and \(\lim_\omega (X,\frac{1}{r_j}d,q)\) are isometric for every sequence \(0 < r_j \nearrow \infty\), every non-principal ultrafilter \(\omega\) on \(\N\) and all points \(p,q\in X\), it follows that if the condition in the previous definition is satisfied for some \(p \in X\) then it holds for every \(p \in X\).
\end{remark}

\begin{remark}\label{str_imm_eqv}
It is easy to see that we could equivalently have required that
\begin{equation}\label{vcond}
\Mass(\iota_\#\bar{R} + \partial \bar{V}) = \Mass(\bar{R}) + \Mass(\partial \bar{V})
\end{equation} 
for every \(\bar{V}\in \Iak{k+1}(X_\omega)\): Assume that \eqref{vcond} holds for every \(\bar{V} \in \Iak{k+1}(X_\omega)\). Let \(\bar{T} \in \Iak{k}(X_\omega)\) be a filling of \(\iota_\#\bar{Z}\) and \(\bar{V} \in \Iak{k+1}(X_\omega)\) a filling of \(\bar{T} - \iota_\#\bar{R}\). Then
\[
\Mass(\bar{T}-\iota_\#\bar{R}) + \Mass(\bar{R}) = \Mass(\partial \bar{V}) + \Mass(\bar{R}) = \Mass(\iota_\#\bar{R} + \partial\bar{R}).
\]
Conversely, assume that \eqref{imm_cond} holds for every filling \(\bar{T}\in \Iak{k}(X_\omega)\) of \(\iota_\#\bar{Z}\). Let \(\bar{V}\in \Iak{k+1}(X_\omega)\) and \(\bar{T}\coloneqq \iota_\#\bar{R} + \partial \bar{V}\). Then 
\[
\Mass(\iota_\# \bar{R} + \partial \bar{V}) = \Mass(\bar{T}) = \Mass(\bar{R})+\Mass(\bar{T}-\iota_\#\bar{R}) = \Mass(\bar{R}) + \Mass(\partial \bar{V})
\]
so \eqref{vcond} holds.
\end{remark}

\begin{example}\label{flat_example}
Let \(k\geq 1\) be an integer and \(F\) a \(k\)-flat in \(X\). Let \(R\in \Zloc{k}(X)\) be any local \(k\)-cycle such that \(F = \spt(R)\). As \(R\) has bounded density at infinity, is conical (with respect to any point in \(F\)) and \(\partial_T \spt(R) = \partial_TF\) is compact, \(R\) can be lifted to \(\bar{R} \in \Iak{k}(C_TX)\); let \(\bar{Z}\coloneqq \partial \bar{R}\in \Zak{k-1}(\partial_TX)\). Now, by \cite[Corollary 1.19]{hks_morse1}, the flat \(F\) is a so called {\it Morse quasiflat} if \(F\) does not bound a flat half-space and the stabilizer of \(F\) in the isometry group of \(X\) acts cocompactly on \(F\), i.e. there exists a compact set \(K\subseteq F\) such that for every \(x \in F\) there exists an isometry \(\varphi: X \rightarrow X\) such that \(\varphi(F) = F\) and \(\varphi(x) \in K\). In that case, \(\bar{Z}\) is strongly immovable as if \(p \in F\), \(0 < r_j \nearrow \infty\) and \(\omega\) is a non-principal ultra-filter on \(\N\), the flat \(F_\omega \coloneqq \lim_\omega (F,\frac{1}{r_j}d,p) \subseteq X_\omega \coloneqq \lim_\omega (X, \frac{1}{r_j}d,p)\) satisfies that if \(\bar{Z}'\in \Zak{k-1}(F_\omega)\) is any cycle and \(\bar{T}' \in \Iak{k}(X_\omega)\) is any chain with \(\partial \bar{T}' = \bar{Z}'\) then it contains the canonical filling of \(\bar{Z}'\) in \(\Iak{k}(F_\omega)\) as a piece (cf. \cite[Theorem 1.12 and Definition 6.10]{hks_morse1}). As \(\bar{Z} \in \Zak{k-1}(F_\omega)\) and \(\bar{R}\) is the canonical filling of \(\bar{Z}\) in \(\Iak{k}(F_\omega)\), it is clear that \(\bar{Z}\) is strongly immovable. Hence, cocompact flats which do not bound flat half-spaces give rise to strongly immovable cycles in the Tits boundary.
\end{example} 

\begin{remark}\label{cone_eq}
In the proof of the next lemma we will need the following observation which is contained in the proof of \cite[Theorem 9.5]{kl_hrh} and the comments thereafter: Let \(k\geq 1\) be an integer, \(\bar{Z}\in \Zak{k-1}(\partial_TX)\) and \(h: C_TX \rightarrow C_TX\), \((\lambda,u)\mapsto \lambda u\) the homothety of the Tits cone \(C_TX\). Then \(\Mass(h_\#(\bb{0,1}\times \bar{Z})) = \frac{1}{k}\Mass(\bar{Z})\). 
For clarity, let us give the proof: By Theorem \ref{cone_ineq}, \(\Mass(h_\#(\bb{0,1}\times \bar{T})) \leq \frac{1}{k}\Mass(\bar{T})\). On the other hand, as 
\[
\Mass(\langle h_\#(\bb{0,1}\times \bar{T}),(d_T)_o,r\rangle) = r^{k-1}\Mass(\bar{T}),
\]
since the homothety \(u\mapsto r u\) takes \(\bar{T}\) to \(\langle h_\#(\bb{0,1}\times \bar{T}),(d_T)_o,r \rangle\), it holds that
\[
\Mass(h_\#(\bb{0,1}\times \bar{T}))\geq \int_0^1 r^{k-1}\cdot \Mass(\bar{T})\: dr = \frac{1}{k}\Mass(\bar{T})
\]
which finishes the proof.
\end{remark}

\begin{lemma}
Let \(k\geq 1\) be an integer and \(\bar{Z}\in \Zak{k-1}(\partial_TX)\) a cycle which is immovable in the strong sense. Then \(\bar{Z}\) is immovable.
\end{lemma}

\begin{proof}
Fix an asymptotic cone \(X_\omega \coloneqq \lim_\omega (X, \frac{1}{r_j},p)\) of \(X\) and let \(\iota: C_TX \rightarrow X_\omega\) denote the embedding from Lemma \ref{asympttits}. Let \(h: C_TX \rightarrow C_TX\), \((\lambda,v)\mapsto \lambda v\), denote the homothety of \(C_TX\). Now, take \(\bar{V}\in \Iak{k}(\partial_TX)\). As \(\bar{Z}\) is immovable in the strong sense, it holds that 
\begin{equation}\label{masstrrt}
\Mass(\bar{T}-\bar{R}) + \Mass(\bar{R}) = \Mass(\bar{T})
\end{equation}
where \(\bar{T}\coloneqq h_\#(\bb{0,1}\times (\bar{Z} + \partial \bar{V})) - \bar{V}\). Now, 
\begin{align*}
\Mass(\bar{T}-\bar{R}) &= \Mass(h_\#(\bb{0,1}\times \partial \bar{V} - \bar{V}) = \Mass(h_\#(\bb{0,1}\times \partial \bar{V})) + \Mass(\bar{V})\\
&= \frac{1}{k}\Mass(\partial \bar{V}) + \Mass(\bar{V})
\end{align*}
where the second last step followed from the fact that \(\|h_\#(\bb{0,1}\times \partial \bar{V})\|(S_{C_TX}(o,1)) = 0\) and \(\bar{V}\) is supported on \(S_{C_TX}(o,1)\). The last step followed from Remark \ref{cone_eq}. Similarly, 
\[
\Mass(\bar{T}) = \frac{1}{k}\Mass(\bar{Z}+\partial \bar{V}) + \Mass(\bar{V})
\]
and it also holds that \(\Mass(\bar{R}) = \frac{1}{k}\Mass(\bar{Z})\). Now \eqref{masstrrt} gives that 
\[
\Mass(\bar{Z}) + \Mass(\partial \bar{V}) = \Mass(\bar{Z}+\partial \bar{V})
\]
which finishes the proof.
\end{proof}

\begin{proposition}\label{immmin}
Let \(k\geq 1\) be an integer, \(\bar{Z}\in \Zak{k-1}(\partial_T X)\) an immovable cycle and \(\bar{R}\in \Iak{k}(C_TX)\) the cone over \(\bar{Z}\) from \(o\). Then \(\bar{R}\) is the unique minimal filling of \(\bar{Z}\).
\end{proposition}

\begin{proof}
Let \(\bar{S} \in \Iak{k}(C_TX)\) be another filling of \(\bar{Z}\) and let \(\bar{V}\) be a filling of \(\bar{S} - \bar{R}\). For every \(r > 0\), \(\langle \bar{R},(d_T)_o,r\rangle \in \Zak{k-1}(\partial_T X)\) is an immovable cycle in \(S(o,r)\) and for almost every \(r > 0\), \(\langle \bar{S},(d_T)_o,r\rangle \in \Iak{k-1}(C_TX)\), \(\langle \bar{V},(d_T)_o,r\rangle \in \Iak{k}(C_TX)\) and then \(\langle \bar{S},(d_T)_o,r\rangle - \langle \bar{R},(d_T)_o,r\rangle = \partial\langle \bar{V},(d_T)_o,r\rangle\) and as \(\langle \bar{R},(d_T)_o,r\rangle\) is an immovable cycle, we have
\[
\Mass(\langle \bar{S},(d_T)_o,r\rangle) = \Mass(\langle \bar{R},(d_T)_o,r\rangle) + \Mass(\partial \langle \bar{V},(d_T)_o,r\rangle) \geq \Mass(\langle \bar{R},(d_T)_o,r\rangle).
\]
It follows from the coarea inequality that
\[
\Mass(\bar{R}) = \int_0^\infty \Mass(\langle \bar{R},(d_T)_o,r\rangle)\: dr \leq \int_0^\infty \Mass(\langle \bar{S},(d_T)_o,r\rangle)\: dr \leq \Mass(\bar{S}).
\]
Now, if \(\Mass(\bar{R}) = \Mass(\bar{S})\) then \(\partial\langle \bar{V},(d_T)_o,r\rangle = 0\) for almost every \(r\) so \(\langle \bar{S},(d_T)_o,r\rangle = \langle \bar{R},(d_T)_o,r\rangle\) for almost every \(r\) and hence \(\bar{R} = \bar{S}\).
\end{proof}

As a corollary, we get the following.

\begin{corollary}\label{immcpt}
Let \(k\geq 1\) be an integer. For every \(m,\varepsilon > 0\) there exists an integer \(N > 0\) such that if \(\bar{Z}\in \Zak{k-1}(\partial_TX)\) is an immovable cycle with \(\Mass(\bar{Z}) \leq m\) and \(\sN\) is an \(\varepsilon\)-separated set in \(\spt(\bar{Z})\), then \(|\sN|\leq N\). In particular, immovable cycles are compactly supported.
\end{corollary}

\begin{proof}
Let \(\bar{Z} \in \Zak{k-1}(\partial_TX)\) be immovable with \(\Mass(\bar{Z}) \leq m\) and \(\bar{R}\in \Iak{k}(C_TX)\) be the cone over it from \(o\). By Lemma \ref{immmin}, \(\bar{R}\) is absolutely minimizing so for every \(x\in \spt(\bar{R})\) and \(0 \leq r < d_T(x,\spt(\bar{Z}))\) it holds that \(\|\bar{R}\|(B(x,r))\geq \theta_k \cdot r^k\). Now, let \(\sN\) be an \(\varepsilon\)-separated set of points in \(\spt(\bar{Z})\). Then \(\frac{1}{2}\sN\) is an \(\frac{\varepsilon}{2}\)-separated set of points in \(\spt(\bar{R})\) so 
\[
\frac{m}{k} \geq \frac{\Mass(\bar{Z})}{k} = \Mass(\bar{R}) \geq \sum_{\xi\in\sN} \|\bar{R}\|(B(\frac{1}{2}\xi,\varepsilon/4)) \geq |\sN|\cdot \theta_k \cdot (\varepsilon'/4)^k
\]
where \(\varepsilon' \coloneqq \min\{1,\varepsilon\}\). This gives an upper bound on \(|\sN|\) depending only on \(m,\,k,\,\varepsilon\) which finishes the proof.
\end{proof}

The following corollary will be important:

\begin{corollary}\label{uni_cpt}
For every integer \(k\geq 1\) and numbers \(m,\varepsilon > 0\) there exists an integer \(N > 0\) such that the following holds: Let \(\bar{Z}\in \Zak{k-1}(\partial_TX)\) be an immovable cycle with \(\Mass(\bar{Z}) \leq m\), \(p\in X\) a point and \(R\in \Zloc{k}(X)\) the cone over \(\bar{Z}\) from \(p\). For every \(r > 0\) it holds that that if \(\sN\) is an \(\varepsilon r\)-separated set of points in \(\spt(R)\cap B(p,r)\) then \(|\sN| \leq N\).
\end{corollary}

\begin{proof}
Let \(\bar{Z}\in\Zak{k-1}(X)\) be an immovable cycle with \(\Mass(\bar{Z})\leq m\) and \(\bar{R} \in \Iak{k}(C_TX)\) the cone over \(\bar{Z}\) from \(o\). The map \(\pi_r: C_TX \rightarrow X\), \(\pi_r(u)\coloneqq \can_p(ru)\) is \(r\)-Lipschitz and surjective from \(\spt(\bar{R})\) onto \(\spt(R)\cap B(p,r)\). Hence if \(\sN\) is an \(\varepsilon r\)-separated set of points in \(\spt(R)\cap B(p,r)\), then there exists an \(\varepsilon\)-separated set \(\bar{\sN}\) of points in \(\spt(\bar{R})\) with \(\pi_r(\bar{\sN}) = \sN\). Now the result follows from Corollary \ref{immcpt}.
\end{proof}

% We also get the following corollary:

% \begin{corollary}\label{minmass}
% For every integer \(k\geq 0\) there is a number \(\bar{m}_k > 0\) such that if \(\bar{Z} \in \Zak{k}(\partial_TX)\) is immovable and \(\bar{Z} \neq 0\) then \(\Mass(\bar{Z}) \geq \bar{m}_k\).
% \end{corollary}

% \begin{proof}
% Assume that \(\bar{Z}\in \Zak{k}(\partial_TX)\) is immovable and non-zero. Let \(\bar{R}\in \Iak{k+1}(C_TX)\) denote the cone over \(\bar{Z}\) from \(o\) in \(C_TX\). Now, \(o \in \spt(\bar{R})\) so by Corollary \ref{lowerbound} and Lemma \ref{immmin}, \(\Mass(\bar{Z}) = (k+1)\cdot \Mass(\bar{R})\geq (k+1)\cdot\theta_{k+1}\). 
% \end{proof}

Now we turn to showing that the immovable cycles and strongly immovable cycles form groups. For that we need the following lemma.

\begin{lemma}\label{cancellation}
Let \(k\geq 0\) be an integer, \(Y\) a complete metric space and \(S,T\in \mathcal{R}_k(Y)\). If \(\sH^k(\set(S)\cap \set(T)) = 0\) then \(\Mass(S+T) = \Mass(S) + \Mass(T)\) and if \(\sH^k(\set(S)\cap \set(T)) > 0\) then \(\Mass(S+T) < \Mass(S) + \Mass(T)\) or \(\Mass(S-T) < \Mass(S) + \Mass(T)\).
\end{lemma}

\begin{proof}
Assume first that \(\sH^k(\set(S)\cap \set(T)) = 0\). Using that \(\spt(S)\ssm \set(S)\) and \(\spt(T)\ssm \set(T)\) are \(\sH^k\)-negligible, we know that \(\sH^k(\spt(S)\cap \spt(T)) = 0\) so we get
\begin{align*}
\Mass(S+T) &= \|S+T\|(Y\ssm \spt(T)) + \|S+T\|(Y\ssm \spt(S))\\
&+ \|S+T\|(\spt(S)\cap \spt(T))\\
&\geq \|S\|(Y\ssm \spt(T)) - \|T\|(Y\ssm \spt(T))\\
&-\|S\|(Y\ssm \spt(S)) + \|T\|(Y\ssm \spt(S))\\
&= \Mass(S) + \Mass(T).
\end{align*}
and hence is \(\Mass(S+T) = \Mass(S) + \Mass(T)\).\\

By Theorem \ref{intr_rep}, there exist Borel functions \(\theta_S: \set(S)\rightarrow [0,\infty)\), \(\theta_T: \set(T)\rightarrow [0,\infty)\) and measurable orientations \(\tau_S,\,\tau_T\) on \(\set(S),\,\set(T)\), respectively, such that 
\[
S(f,\pi_1,\ldots,\pi_k) = \int_{\set(S)} f(x)\theta_S(x) \langle \Alt_k (D\pi(x)|_{T_xS}),\tau_S\rangle \: d\sH^k(x)
\]
and
\[
T(f,\pi_1,\ldots,\pi_k) = \int_{\set(T)} f(x)\theta_T(x) \langle \Alt_k (D\pi(x)|_{T_xT}),\tau_T\rangle \: d\sH^k(x)
\]
for all Lipschitz functions \(f\), \(\pi_1,\ldots,\pi_k\) with \(f\) bounded. We extend the functions \(\theta_S,\,\theta_T\) to \(\set(S)\cup \set(T)\) by letting \(\theta_S(x) = 0\) for \(x\notin \set(S)\) and \(\theta_T(x) = 0\) for \(x\notin \set(T)\). We define an orientation on \(\set(S)\cup \set(T)\) by letting \(\tau(x) = \tau_S(x)\) for \(x\in \set(S)\) and \(\tau(x) = \tau_T(x)\) for \(x\in \set(T)\ssm \set(S)\). As for \(\sH^k\)-a.e. \(x \in \set(S)\cap \set(T)\), the tangent spaces to \(S\) and \(T\) at \(x\) agree, it holds that
\[
T(f,\pi_1,\ldots,\pi_k) = \int_{\set(T)} f(x)\Tilde{\theta}_T(x) \langle \Alt_k (D\pi(x)_{T_xT}),\tau\rangle \: d\sH^k(x)
\]
where \(\Tilde{\theta}_T(x) = -\theta_T(x)\) if \(x\in \set(S)\cap\set(T)\) and \(\tau_S(x) = -\tau_T(x)\) and \(\Tilde{\theta}_T(x) = \theta_T(x)\) otherwise. Now, let \(f,\pi_1,\ldots,\pi_k\) be Lipschitz functions, \(f\) bounded. Then 
\[
(S+T)(f,\pi_1,\ldots,\pi_k) = \int_{\set(S)\cup \set(T)} f(x)(\theta_S(x)+\Tilde{\theta}_T(x))\langle \Alt_k (D\pi(x)|_{T_xS}), \tau\rangle\: d\sH^k(x)
\]
so by \eqref{area_factor_ineq},
\[
\begin{split}
&|(S+T)(f,\pi_1,\ldots,\pi_k)| \\
&\quad\leq \int_{\set(S)\cup \set(T)} |f(x)|\cdot |\theta_S(x)+\Tilde{\theta}_T(x)|\cdot \lambda(x)\: d\sH^k(x) \cdot \prod_{i=1}^k \Lip(\pi_i)
\end{split}
\]
where \(\lambda(x)\) is defined as the area factor of the tangent space to \(\set(S)\cup \set(T)\) when it exists, which is for \(\sH^k\)-a.e. \(x\in \set(S)\cup \set(T)\). This shows that \(\|S+T\| \leq |\theta_S + \Tilde{\theta}_T|\lambda \sH^k\). Note that, in exactly the same way, we get for \(-T\) that \(\|S-T\| \leq |\theta_S + \Tilde{\theta}_{-T}|\lambda \sH^k\) where \(\Tilde{\theta}_{-T}(x) = -\Tilde{\theta}_T(x)\) for \(x\in \set(S)\cap \set(T)\) and \(\Tilde{\theta}_{-T}(x) =\Tilde{\theta}_T(x)\) else. Now, if \(\Mass(S+T) = \Mass(S) + \Mass(T)\) then \(|\theta_S(x) + \theta_T(x)| = |\theta_S(x) + \Tilde{\theta}_T(x)|\) for \(\sH^k\)-a.e. \(x\) which again means that \(\theta_T(x) > 0\) for \(\sH^k\)-a.e. \(x\in \set(S)\cup \set(T)\). But then \(|\theta_S + \Tilde{\theta}_{-T}| < |\theta_S + \theta_T|\) on a set of positive \(\sH^k\)-measure so \(\Mass(S-T) < \Mass(S) + \Mass(T)\). This finishes the proof.
\end{proof}

Now it is easy to prove the following.

\begin{proposition}\label{str_imm_subgrp}
For every integer \(k\geq 1\), the set of immovable cycles and the set of strongly immovable cycles in \(\Zak{k-1}(\partial_TX)\) are subgroups.
\end{proposition}

\begin{proof}
Let \(\bar{Z}_1,\bar{Z}_2 \in \Zak{k-1}(\partial_TX)\) be immovable cycles. If \(\bar{Z}_1+\bar{Z}_2\) is not immovable, then there exists \(\bar{T}\in\Iak{k}(\partial_TX)\) such that 
\[
\Mass(\bar{Z}_1+\bar{Z}_2 + \partial \bar{T}) < \Mass(\bar{Z}_1 + \bar{Z}_2) + \Mass(\bar{T})
\]
so by Lemma \ref{cancellation}, \(\sH^{k-1}(\set(\bar{Z}_1+\bar{Z}_2)\cap \set(\partial T)) > 0\) and as \(\set(\bar{Z}_1+\bar{Z}_2) \subseteq \set(\bar{Z}_1)\cup \set(\bar{Z}_2)\), it follows that \(\sH^{k-1}(\set(\bar{Z}_1)\cap \set(\partial T)) > 0\) or \(\sH^{k-1}(\set(\bar{Z}_2)\cap \set(\partial T)) > 0\) and hence, again by Lemma \ref{cancellation}, \(\Mass(\bar{Z}_1+\partial T) < \Mass(\bar{Z}_1) + \Mass(\partial T)\), \(\Mass(\bar{Z}_1 - \partial T) < \Mass(\bar{Z}_1) + \Mass(\partial T)\), \(\Mass(\bar{Z}_2+\partial T) < \Mass(\bar{Z}_2) + \Mass(\partial T)\) or \(\Mass(\bar{Z}_2 - \partial T) < \Mass(\bar{Z}_1) + \Mass(\partial T)\) which contradicts that \(\bar{Z}_1\) and \(\bar{Z}_2\) are immovable.\\

The strongly immovable case is almost identical: Let \(\bar{Z}_1,\bar{Z}_2 \in \Zak{k-1}(\partial_TX)\) be strongly immovable, \(0 < r_j \nearrow \infty\), \(j\in \N\), a sequence, \(p \in X\) a point, \(\omega\) a non-principal ultrafilter on \(\N\), \(X_\omega \coloneqq \lim_\omega (X,\frac{1}{r_j}d,p)\) and \(\iota: C_TX \rightarrow X_\omega\) the isometric embedding from Lemma \ref{asympttits}. Let \(\bar{R}_1,\bar{R}_2 \in \Iak{k}(C_TX)\) be the cones from \(o\in C_TX\) over \(\bar{Z}_1,\bar{Z}_2\), respectively. Now, if \(\bar{Z}_1+\bar{Z}_2\) is not strongly immovable, then there exists \(\bar{V}\in \Iak{k+1}(X_\omega)\) such that 
\[
\Mass(\iota_\#\bar{R}_1 + \iota_\#\bar{R}_2 + \partial \bar{V}) < \Mass(\iota_\#\bar{R}_1 + \iota_\#\bar{R}_2) + \Mass(\partial \bar{V}),
\]
by Remark \ref{str_imm_eqv}. Then by Lemma \ref{cancellation}, \(\sH^k(\set(\iota_\#\bar{R}_1+\iota_\#\bar{R}_2)\cap \set(\bar{V})) > 0\) so \(\sH^k(\set(\iota_\#\bar{R}_1)\cap \set(\bar{V})) > 0\) or \(\sH^k(\set(\iota_\#\bar{R}_2)\cap \set(\bar{V})) > 0\) as \(\set(\bar{R}_1 + \bar{R}_2) \subseteq \set(\bar{R}_1) \cup \set(\bar{R}_2)\). But then, again by Lemma \ref{cancellation}, it follows that \(\Mass(\iota_\#\bar{R}_1 + \partial \bar{V}) < \Mass(\bar{R}_1) + \Mass(\partial \bar{V})\), \(\Mass(\iota_\#\bar{R}_1 - \partial \bar{V}) < \Mass(\bar{R}_1) + \Mass(\partial \bar{V})\), \(\Mass(\iota_\#\bar{R}_2 + \partial \bar{V}) < \Mass(\bar{R}_2) + \Mass(\partial \bar{V})\) or \(\Mass(\iota_\#\bar{R}_2 - \partial \bar{V}) < \Mass(\bar{R}_2) + \Mass(\partial \bar{V})\) which contradicts that \(\bar{Z}_1\) and \(\bar{Z}_2\) are strongly immovable.
\end{proof}

\section{Solving the asymptotic Plateau problem}\label{solving}

Now we turn to showing that the asymptotic Plateau problem admits a solution for cycles in the Tits boundary which are immovable in the strong sense. We will adapt the method of Kleiner and Lang in \cite{kl_hrh}, which they use to solve the asymptotic Plateau problem for top dimensional cycles. A key technical ingredient in their method is Wenger's sub-Euclidean isoperimetric inequality \cite[Theorem 1.2]{wenger_asymptotic}, which allows them to conduct an induction on scales argument to prove \cite[Proposition 4.5]{kl_hrh} from which they can then easily prove existence. We begin by proving the following proposition, which will serve as a substitute for Wenger's sub-Euclidean isoperimetric inequality and allow us to conduct induction on scales. The proof adopts fundamental ideas from the proof of Wenger's sub-Euclidean isoperimetric inequality, in particular the application of Gromov's compactness theorem together with the compactness theorem of Ambrosio and Kirchheim, and then the use of the auxiliary space \(Y\) to relate convergence in the compact space which Gromov's compactness theorem provides, to convergence in the original space – a trick which Wenger originally introduced in the proof of \cite[Theorem 1.3]{wenger_filling}.

\begin{proposition}\label{reliso}
Let \(X\) be a proper Hadamard space, \(k\geq 1\) an integer, \(p\in X\) a point, \(\bar{Z}\in \Zak{k-1}(\partial_TX)\) a cycle which is immovable in the strong sense, \(R\in \Zloc{k}(X)\) the cone from \(p\in X\) over \(\bar{Z}\) and \(\Theta\coloneqq \Theta_\infty(R)\). For every \(C,\varepsilon,a > 0\) there exists \(\bar{r} = \bar{r}(C,\varepsilon,a) > 0\) such that the following holds: Let \(r\geq \bar{r}\), and \(S\in \Iloc{k}(X)\) be a minimal chain with \(\|S\|(B(p,s)) \leq \Theta \cdot s^k\) for every \(s\geq 0\). Assume that \(\langle S,d_p,r\rangle \in \Zak{k-1}(X)\) with \(\Mass(\langle S,d_p,r\rangle) \leq C\cdot r^{k-1}\) and that \(V \in \Iak{k+1}(X)\) is absolutely minimizing with \(\Mass(V)\leq C\cdot r^{k+1}\), \(\spt(S-R - \partial V)\) at distance \(\geq (1+a)\cdot r\) from \(p\), \(\langle V,d_p,r\rangle \in \Iak{k}(X)\) and \(\Mass(\langle V,d_p,r\rangle) \leq C\cdot r^k\). Then there exists \(W\in \Iak{k}(X)\) with \(\spt(S - R - \partial W)\) at distance \(\geq (1-a)\cdot r\) from \(p\) and \(\Mass(W) < \varepsilon \cdot r^{k+1}\).
\end{proposition}

\begin{proof}
Assume for a contradiction that there exist \(C,a,\varepsilon > 0\) and sequences \(0 < r_j\nearrow \infty\) and \(S_j\in \Iloc{k}(X)\), \(j\in \N\), such that for every \(j\in \N\) the following holds: \(\|S_j\|(B(p,s)) \leq \Theta\cdot s^k\) for every \(s \geq 0\), \(\langle S_j,d_p,r_j\rangle \in \Zak{k-1}(X)\) and \(\Mass(\langle S_j,d_p,r_j\rangle) \leq C\cdot r_j^k\); there is an absolutely minimizing chain \(V_j\in \Iak{k+1}(X)\) with \(\Mass(V_j) \leq C\cdot r_j^{k+1}\), \(\spt(S_j-R-\partial V_j)\) at distance \(\geq (1+a)\cdot r_j\) from \(p\) and \(\Mass(\langle V_j,d_p,r_j\rangle) \leq C\cdot r_j^k\); there exists no \(W\in \Iak{k+1}(X)\) with \(\spt(S_j-R-\partial W)\) at distance \(> (1-a)\cdot r_j\) from \(p\) and \(\Mass(W) < \varepsilon\cdot r_j^{k+1}\).\\

Consider now the spaces \(Z_j \coloneqq (\spt(V_j)\cap B(p,r_j), \frac{1}{r_j}d)\). 
\begin{claim}\label{precptclaim}
The spaces \(Z_j\), \(j\in \N\), are uniformly precompact, i.e. \(\sup_{j\in \N}\diam(Z_j) < \infty\) and for every \(\varepsilon > 0\) there exists an \(N\) such that if \(\sN\) is an \(\varepsilon\)-separated set in \(Z_j\) then \(|\sN|\leq N\).
\end{claim}
\begin{proof}[Proof of Claim \ref{precptclaim}]
It is clear that \(\sup_{j\in \N}\diam(Z_j) < \infty\). Let \(\varepsilon > 0\) and \(\sN\) be an \(\varepsilon\)-separated set in \(Z_j\). Without loss of generality, we assume that \(\varepsilon < \frac{3}{2}\cdot a\). Let 
\begin{align*}
\sN' &\coloneqq \left\{x\in \sN \;\big|\; d(x,\spt(\partial V_j)) > \frac{\varepsilon}{3}\cdot r_j\right\},\\
\sN'' &\coloneqq \left\{x \in \sN \;\big|\; d(x,\spt(S_j)) \leq \frac{\varepsilon}{3}\cdot r_j\right\},\\
\sN''' &\coloneqq \left\{x\in \sN \;\big|\; d(x,\spt(R)) \leq \frac{\varepsilon}{3}\cdot r_j\right\}.
\end{align*}
As \(\sN \subseteq \sN' \cup \sN'' \cup \sN'''\) it suffices to bound the cardinality of each of the sets \(\sN',\,\sN'',\,\sN'''\) independently of \(j\) in order to obtain a bound on \(|\sN|\) independent of \(j\). Now,
\[
C\cdot r_j^{k+1} \geq \|V_j\|(B(p,r_j)) \geq \sum_{x\in \sN'} \|V_j\|(B(x,\frac{\varepsilon}{3}\cdot r_j)) \geq |\sN'|\cdot \theta_{k+1} \cdot \left(\frac{\varepsilon}{3}\cdot r_j\right)^{k+1}
\]
which gives an upper bound on \(|\sN'|\) independent of \(j\). For each point \(x\in \sN''\) we pick a point \(x' \in \spt(S_j)\) with \(d(x,x') \leq \frac{\varepsilon}{3}\cdot r_j\). Then if \(x,y\in \sN''\) are distinct, \(d(x',y') > \frac{\varepsilon}{3}\cdot r_j\). Further, for every \(x\in \sN''\), \(d(p,x') \leq (1+\frac{\varepsilon}{3})\cdot r_j\) so from the assumption that \(\varepsilon < \frac{3}{2}\cdot a\), and the fact that \(d(p,\spt(\partial S_j)) \geq (1+a)\cdot r_j\), it follows that \(d(x',\spt(\partial S_j)) > \frac{\varepsilon}{3}\cdot r_j\). Hence we can estimate:
\begin{align*}
C\cdot (1+a)^k\cdot r_j^k &\geq \|S_j\|(B(p,(1+a)\cdot r_j)) \geq \sum_{x\in \sN''} \|S_j\|(B(x,\frac{\varepsilon}{3}\cdot r_j))\\
&\geq |\sN''|\cdot \theta_k \cdot \left(\frac{\varepsilon}{3}\cdot r_j\right)^k
\end{align*}
which gives an upper bound on \(|\sN''|\) independent of \(j\). Finally, for each \(x\in\sN'''\) we pick \(x'\in \spt(R)\) with \(d(x,x') < \frac{\varepsilon}{3}\cdot r_j\) and obtain an \(\left(\frac{\varepsilon}{3}\cdot r_j\right)\)-separated set in \(\spt(R)\cap B(p,(1+\frac{a}{2})\cdot r_j)\) whose cardinality is bounded independent of \(j\) by Corollary \ref{uni_cpt}.
\end{proof}

Now, by Gromov's compactness theorem, there exists a compact metric space \((Z,d_Z)\) together with isometric embeddings \(\varphi_j:Z_j \rightarrow Z\). For every \(j\in \N\), let \(\iota_j: (X,d)\rightarrow (X,\frac{1}{r_j}d)\) denote the identity map. By the compactness theorem of Ambrosio and Kirchheim \cite[Theorem 5.2]{ambrosio_kirchheim}, we may assume that there are \(\bar{V} \in \Iak{k+1}(Z)\) and \(\bar{S},\bar{R}\in \Iak{k}(Z)\) and that after possibly passing to a subsequence, \((\varphi_j\circ \iota_j)_\# (V_j\res B(p,r_j)) \rightarrow \bar{V}\), \((\varphi_j\circ \iota_j)_\# (S_j\res B(p,r_j)) \rightarrow \bar{S}\) and \((\varphi_j\circ \iota_j)_\# (R\res B(p,r_j)) \rightarrow \bar{R}\) as \(j\rightarrow \infty\). By Proposition \ref{embedding_lemma}, \(\bar{V}\in \Iak{k+1}(X_\omega)\) and \(\bar{S},\bar{R}\in \Zak{k}(X_\omega)\) where \(X_\omega \coloneqq \lim_\omega (X,\frac{1}{r_j}d,p)\) and \(\omega\) is some fixed non-principal ultrafilter on \(\N\).

\begin{claim}\label{lift_asympt}
Using the embedding \(\iota: C_TX \rightarrow X_\omega\) from Lemma \ref{asympttits} to view \(C_TX\) as a subspace of \(X_\omega\), it holds that \(\bar{R}\) is the lift of \(R\) to \(C_TX\).
\end{claim}

\begin{proof}[Proof of Claim \ref{lift_asympt}]
Let \(K\coloneqq [0,1]\cdot \spt(\bar{Z}) \subseteq C_TX\). Note that for every \(r \geq 0\), \(\spt(R)\cap B(p,r) \subseteq \can_p(r K)\). Now, let \([(x_j)_{j\in\N}]\in \lim_\omega(\spt(R)\cap B(p,r_j),\frac{1}{r_j}d,p)\). For every \(j\in \N\), there is \(u_j\in K\) such that \(\can_p(r_ju_j) = x_j\). As \(K\) is compact since \(\spt(\bar{Z})\) is compact, we can define \(u\coloneqq \lim_\omega u_j\). It holds that
\begin{align*}
d_\omega(\iota(u),[(x_j)_{j\in \N}]) &= \lim_\omega \frac{1}{r_j}d(\can_p(r_ju),x_j) = \lim_\omega \frac{1}{r_j} d(\can_p(r_ju),\can_p(r_ju_j))\\
&\leq \lim_\omega \lim_{r\rightarrow \infty}\frac{1}{r}d(\can_p(ru),\can_p(ru_j)) =\lim_\omega d_T(u,u_j)\\
&= 0
\end{align*}
so \(\iota(u) = [(x_j)_{j\in \N}]\) and hence is \(\spt(\bar{R})\subseteq \iota(K) \subseteq \iota(C_TX)\). In the proof of \cite[Theorem 9.3]{kl_hrh}, the lift of \(R\) to \(C_TX\) is obtained by embedding the spaces \((\spt(R)\cap B(p,r),\frac{1}{r}d)\) into a common compact metric space and showing that the push-forwards of the currents \(R\res B(p,r)\), \(r > 0\), where \(R\res B(p,r)\) is viewed as an element of \((X,\frac{1}{r}d)\), converge in the flat topology as \(r\rightarrow \infty\), to a current which is supported on an isometric copy of \(K\subseteq C_TX\). The proof of the claim now follows from \cite[Theorem 1.3]{wenger_compact}.
\end{proof}

From the condition that \(\|S_j\|(B(p,s))\leq \Theta\cdot s^k\) for every \(s \geq 0\), and lower semi-continuity of mass on open sets, it follows that \(\|\bar{S}\|(U(p_\omega,r)) \leq \Theta\cdot r^k\) for every \(0\leq r\leq 1\) and hence also \(\|\bar{S}\|(B(p_\omega,r)) \leq \Theta\cdot r^k\) for all \(0 \leq r < 1\). By lower semi-continuity of mass \(\Mass(\bar{S})\leq \Theta\) so \(\|\bar{S}\|(B(p_\omega,r)) \leq \Theta\cdot r^k\) for all \(0 \leq r \leq 1\). By Theorem \ref{lift}, \(\Mass(\bar{R}) = \Theta\) and since \(\bar{R} \res B(p_\omega,r) = h_{r\#}\bar{R}\) where \(h_r: C_TX \rightarrow C_TX\), \(u\mapsto ru\), it follows that \(\Mass(\bar{R}\res B(p_\omega,r)) = \Theta \cdot r^k\) so \(\Mass(\bar{S}\res B(p_\omega,r)) \leq \Mass(\bar{R}\res B(p_\omega,r))\) for every \(0\leq r\leq 1\). Let \(\bar{T}\coloneqq \partial \bar{V} + \bar{R}-\bar{S}\). Using that \(\bar{T}\) is supported on \(S(p_\omega,1)\) and the strong immovability of \(\bar{Z}\), we get that for almost every \(0 \leq r < 1\) it holds that
\begin{align*}
\|\bar{S}-\bar{R}\|(B(p_\omega,r)) &= \Mass(\bar{S}+\bar{T}-\bar{R})-\|\bar{S}+\bar{T}-\bar{R})\|(B(p_\omega,r)^c)\\
&= \Mass(\bar{S}+\bar{T}) - \|\bar{S}+\bar{T}-\bar{R}\|(B(p_\omega,r)^c) - \Mass(\bar{R})\\
&\leq \Mass(\bar{S}+\bar{T}) - \|\bar{S}+\bar{T}\|(B(p_\omega,r)^c) - \Mass(\bar{R}) + \|\bar{R}\|(B(p_\omega,r)^c)\\
&\leq \|\bar{S}\|(B(p_\omega,r)) - \|\bar{R}\|(B(p_\omega,r))\\
&\leq 0
\end{align*}
so \(\bar{S}\res B(p_\omega,r) = \bar{R}\res B(p_\omega,r)\). It follows that \(\|\bar{S}\|(B(p_\omega,r)) = \Theta\cdot r^k\) for every \(0 \leq r < 1\) and as \(\Mass(\bar{S})\leq \Mass(\bar{R}) = \Theta\) it follows that \(\|\bar{S}\|(S(p_\omega,1)) = 0\) so \(\bar{S} = \bar{R}\).\\

Now we do as Wenger in the proofs of \cite[Theorem 1.2]{wenger_asymptotic} and \cite[Theorem 1.3]{wenger_filling}: Consider the metric space \(Y \coloneqq \bigsqcup_{j\in \N} (X,\frac{1}{r_j}d)\) endowed with the metric 
\[
d_Y(x,y) \coloneqq 
\begin{cases}
\frac{1}{r_j}d(x,y) & \text{if } x,y \in (X,\frac{1}{r_j}d)\\
\frac{1}{r_i}d(x,p) + 3 + \frac{1}{r_j}d(y,p) & \text{if } x\in (X,\frac{1}{r_i}d),\,y\in (X,\frac{1}{r_j}d)
\end{cases}
.
\]
For every \(j\in \N\), let \(\psi_j: (X,d) \rightarrow (X,\frac{1}{r_j}d)\subseteq Y\). We are going to show that \(E_j\coloneqq \psi_{j\#}((S_j-R)\res B(p,r_j)) \rightarrow 0\) as \(j\rightarrow \infty\). So assume for a contradiction that there exists an \(\eta > 0\) together with Lipschitz functions \(f,\pi_1,\ldots,\pi_k\) on \(Y\), \(f\) bounded, such that after possibly throwing out a finite number of elements of the sequence, \(E_j(f,\pi_1,\ldots,\pi_k)\geq \eta\) for every \(j \in \N\). For every \(j\in \N\), define on \(\varphi_j(Z_j)\subseteq Z\) the functions \(f_j\coloneqq f\circ \psi_j\circ (\varphi_j\circ \iota_j)^{-1}\) and \(\pi_{ij}\coloneqq \pi_i \circ \psi_j\circ (\varphi_j\circ \iota_j)^{-1}\), \(i\in \{1,\ldots,k\}\). By McShane's extension theorem there exist extensions \(\hat{f}_j\),\(\hat{\pi}_{ij}\) of \(f\),\(\pi_{ij}\) to \(Z\) with the same Lipschitz constants. By Arzelà-Ascoli, there exist Lipschitz functions \(\hat{f}\),\(\hat{\pi}_i\) such that after possibly passing to a subsequence, \(\hat{f}_j\rightarrow \hat{f}\) and \(\hat{\pi}_{ij}\rightarrow \hat{\pi}_i\) uniformly as \(j\rightarrow \infty\). Write \(F_j\coloneqq(\varphi_j\circ \iota_j)_\#((S_j-R)\res B(p,r_j))\). Now,
\begin{align*}
0 &= (\bar{S}-\bar{R})(\hat{f},\hat{\pi}_1,\ldots,\hat{\pi}_k)\\
&= \lim_{j\rightarrow \infty} F_j(\hat{f},\hat{\pi}_1,\ldots,\hat{\pi}_k)\\
&= \lim_{j\rightarrow \infty} \Big(F_j(\hat{f}_j,\hat{\pi}_{1j},\ldots,\hat{\pi}_{kj}) + F_j(\hat{f}-\hat{f}_j,\hat{\pi}_1,\ldots,\hat{\pi}_k)\\
&\qquad\qquad + F_j(\hat{f}_j,\hat{\pi}_1,\ldots,\hat{\pi}_k) - F_j(\hat{f}_j,\hat{\pi}_{1j},\ldots,\hat{\pi}_{kj})\Big)\\
&\geq \eta - \limsup_{j\rightarrow \infty} \left( \prod_{i=1}^k \Lip(\hat{\pi}_i) \cdot \int_Z |\hat{f}-\hat{f}_j|\: d\|F_j\|\right)\\
& \qquad - \limsup_{j\rightarrow \infty} \left(\sum_{i=1}^k \int_Z |\hat{f}_j|\cdot |\hat{\pi}_i-\hat{\pi}_{ij}|\: d\|\partial F_j\| + \Lip(\hat{f}_j) \cdot \int_Z |\hat{\pi}_i-\hat{\pi}_{ij}|\: d\|F_j\|\right)\\
&\geq \eta
\end{align*}
which is a contradiction. The second last step used that
\[
F_j(\hat{f}_j,\hat{\pi}_{1j},\ldots,\hat{\pi}_{kj}) = E_j(f,\pi_1,\ldots,\pi_k) \geq \eta,
\]
the definition of mass and \cite[Proposition 5.1]{ambrosio_kirchheim}. The last step relied on the estimate
\[
\Mass(\partial F_j) \leq \Mass(\psi_{j\#}\langle S_j,d_p,r_j\rangle) + \Mass(\psi_{j\#}\langle R,d_p,r_j\rangle) \leq C + k\cdot \Theta.
\]
Now, as \(E_j \rightarrow 0\) weakly, \(\partial E_j \rightarrow 0\) weakly and as \(\sup_{j\in \N}\Mass(\partial E_j) < \infty\), it follows from \cite[Theorem 1.4]{wenger_flat} that \(\FillVol(\partial E_j)\rightarrow 0\) as \(j\rightarrow \infty\) since \(Y\) satisfies local cone type inequalities. Hence we may pick fillings \(T_j'\in \Iak{k}(Y)\) of \(\partial E_j\) with \(\Mass(T_j')\rightarrow 0\). We may obviously assume that \(T_j'\) is supported on \((X,\frac{1}{r_j}d)\) so \(T_j'\) corresponds to a filling \(T_j\) of \(\partial ((S_j-R)\res B(p,r_j))\). By possibly throwing out a finite number of elements of our sequence, we may assume that \(\Mass(T_j') < \theta_k \cdot a^k\) for every \(j\in \N\). Then it follows from Corollary \ref{lowerbound} that \(\spt(T_j)\) is at distance \(< a\cdot r_j\) from \(S(p,r_j)\) and hence at distance \(>(1-a)\cdot r_j\) from \(p\). Finally, as \(E_j - T_j' \rightarrow 0\) as \(j\rightarrow \infty\), again by \cite[Theorem 1.4]{wenger_flat}, for every \(\varepsilon > 0\), there exists an integer \(N\) such that if \(j\geq N\) then there exists a filling \(W_j'\in \Iak{k+1}(Y)\) of \(E_j - T_j'\) with \(\Mass(W_j') < \varepsilon\). As for \(T_j'\), we may assume that \(W_j'\) is supported on \((X,\frac{1}{r_j}d)\) so it corresponds to a \(W_j\in \Iak{k+1}(X)\) with \(\Mass(W_j) < \varepsilon \cdot r_j^{k+1}\). This is a contradiction.
\end{proof}

We will need the following lemma, which is essentially stated as part of \cite[Theorem 5.6]{kl_hrh}. There they say that it can be proven by a slicing argument and also refer to \cite[Proposition 4.5]{kl_hrh}, from which the claim follows, but as that proposition is not at our disposal and they do not write out the slicing argument, we provide the proof here.

\begin{lemma}\label{density_lemma}
\emph{(cf. \cite[Theorem 5.6]{kl_hrh})}
Let \(X\) be a proper Hadamard space, \(k\geq 1\) an integer, \(R \in \Zloc{k}(X)\) with \(\Theta_\infty(R) < \infty\) and \(S\in \Zloc{k}(X)\) a minimizer with \(F_\infty(S-R) = 0\). Then \(\Theta_\infty(S)\leq \Theta_\infty(R)\).
\end{lemma}

\begin{proof}
We use a slicing argument as is hinted at in the proof of \cite[Theorem 5.6]{kl_hrh} (cf. also the proof of \cite[Proposition 4.5]{kl_hrh}): Let \(\varepsilon > 0\). There exists an \(r_\varepsilon > 0\) such that if \(r \geq r_\varepsilon\) then \(\|R\|(B(p,r))\leq (\Theta +\varepsilon)\cdot r^k\) and there exists \(V \in \Iak{k+1}(X)\) with \(\spt(S - R - \partial V) \cap B(p,(1+\varepsilon)\cdot r) = \emptyset\) and \(\Mass(V) < \varepsilon^2\cdot r^k\). Fix \(r \geq r_\varepsilon\) and \(V \in \Iak{k+1}(X)\) with \(\spt(S - R - \partial V) \cap B(p,(1+\varepsilon)\cdot r) = \emptyset\) and \(\Mass(V) < \varepsilon^2\cdot r^k\). By the slicing theorem, there exists an \(s \in (r,(1+\varepsilon)\cdot r]\) such that \(\langle V, d_p, s\rangle \in \Iak{k}(X)\) and \(\Mass(\langle V,d_p,s\rangle) < \varepsilon\cdot r^k\). Then, by minimality of \(S\), 
\[
\|S\|(B(p,r)) \leq \|S\|(B(p,s)) \leq \Mass(\langle V,d_p,s\rangle + R \res B(p,s)) \leq (\Theta + 2\varepsilon)\cdot r^k.
\]
Since this holds for arbitrarily large \(r\), \(\Theta_\infty(S) \leq \Theta + 2\varepsilon\) and as \(\varepsilon > 0\) is arbitrary, \(\Theta_\infty(S)\leq \Theta\).
\end{proof}

\begin{remark}\label{Smassbound}
Let \(R\in \Zloc{k}(X)\) be conical with respect to \(p \in X\), \(\Theta\coloneqq \Theta_\infty(R) < \infty\). Then by \cite[Lemma 7.2(2)]{kl_hrh}, the quotient \(\|R\|(B(p,r))\cdot r^{-k}\) is non-decreasing in \(r > 0\) so \(\|R\|(B(p,r))\leq \Theta\cdot r^k\) for every \(r > 0\). Let \(s > 0\) and \(S_s\) be a minimal filling of \(\langle R,d_p,s\rangle\). Then, by the monotonicity formula,
\[
\|S_s\|(B(p,r)) \leq \frac{\|S_s\|(B(p,s))}{s^k}\cdot r^k \leq  \frac{\|R\|(B(p,s))}{s^k}\cdot r^k \leq \Theta\cdot r^k
\]
for every \(0 \leq r \leq s\).
Also, if \(S \in \Zloc{k}(X)\) is a minimizer with \(F_\infty(S-R) = 0\), then by Lemma \ref{density_lemma} and the monotonicity formula, it holds that \(\|S\|(B(p,r))\leq \Theta\cdot r^k\) for every \(r\geq 0\).
\end{remark}

Now we prove a proposition which will play the role of \cite[Proposition 4.5]{kl_hrh} and whose proof uses a similar induction on scales argument.

\begin{proposition}\label{partial_filling}
Let \(X\) be a proper Hadamard space, \(k\geq 1\) an integer, \(\bar{Z}\in \Zak{k}(\partial_T X)\) a strongly immovable cycle and \(R\in \Zloc{k}(X)\) the cone over \(\bar{Z}\) from \(p\). Write \(R_r\coloneqq R\res B(p,r)\). For every \(\varepsilon > 0\) there exists \(r_\varepsilon > 0\) such that if \(r\geq r_\varepsilon\) then 
\begin{enumerate}[label=\emph{(\roman*)}]
    \item \(F_{p,r}(S_s - R_s) < \varepsilon\) for every \(s \geq r\) and \(S_s\in \Iak{k}(X)\) minimal with \(\partial S_s = \partial R_s\).
    \item \(F_{p,r}(S-R) < \varepsilon\) for every \(S \in \Zloc{k}(X)\) minimal with \(F_\infty(S-R) = 0\).
\end{enumerate}
\end{proposition}

\begin{proof}
Let \(\varepsilon > 0\). Without loss of generality, we assume that \(\varepsilon \leq 2\cdot \left(\frac{8}{3}\right)^k\cdot \Theta\) where \(\Theta \coloneqq \Theta_\infty(R)\). Let \(r_\varepsilon \coloneqq \bar{r}\left(C,(\frac{3}{32})^{k+1}\cdot \varepsilon,\frac{1}{2}\right)\) where \(C\coloneqq 9\cdot \left(\frac{4}{3}\right)^k\cdot\Theta\) and \(\bar{r}\) is the function from Proposition \ref{reliso}.\\

(i) Let \(s \geq r \geq r_\varepsilon\). We let \(s_j\coloneqq 4^{-j}\cdot s\) for \(j\geq 1\) and \(V_0\in \Iak{k+1}(X)\) be a minimal filling of \(S_s - R_s\). Then \(\Mass(V_0) < \left(\frac{3}{32}\right)^{k+1}\cdot \varepsilon\cdot s_0^{k+1} < \frac{\varepsilon}{4^{k+1}}s_0^{k+1}\). Now assume that we are given \(V_j\in \Iak{k+1}(X)\) with \(\spt(S_s-R_s - \partial V_j)\) at distance \(> s_j\) from \(p\) and \(\Mass(V_j) < \frac{\varepsilon}{4^{k+1}}s_j^{k+1}\). If \(4r < s_j\) we find \(V_{j+1}\in \Iak{k+1}(X)\) in the following way: By the slicing theorem, there exists a measurable set \(A \subseteq (\frac{1}{2}s_j,\frac{2}{3}s_j)\) with \(|A|\geq \frac{1}{9}s_j\) and such that for every \(s'\in A\) it holds that \(\langle S_s,d_p,s'\rangle \in \Zak{k-1}(X)\) and
\begin{equation}\label{sslicebound}
\Mass(\langle S_s,d_p,s'\rangle) \leq 18\cdot \Theta \cdot \left(\frac{2}{3}\right)^ks_j^{k-1} \leq C\cdot (s')^{k-1}.
\end{equation}
Again, by the slicing theorem, there exists a measurable set \(B\subseteq (\frac{1}{2}s_j,\frac{2}{3}s_j)\) with \(|B|\geq \frac{1}{9}s_j\) and such that for every \(s' \in B\) it holds that \(\langle V_j,d_p,s'\rangle \in \Iak{k}(X)\) and 
\begin{equation}\label{vslicebound}
\Mass(\langle V_j,d_p,s'\rangle) \leq 18\cdot \frac{\varepsilon}{4^{k+1}}\cdot s_j^k\leq C\cdot (s')^k.
\end{equation}
As \(|A|,|B|\geq \frac{1}{9}s_j\), \(|A|+|B| > \frac{1}{6}s_j = |(\frac{1}{2}s_j,\frac{2}{3}s_j)|\) so there exists \(s'\in (\frac{1}{2}s_j,\frac{2}{3}s_j)\) such that \eqref{sslicebound} and \eqref{vslicebound} hold. Then as \(\spt(S_s-R_s-\partial V_j)\) is at distance \(> s_j \geq \frac{3}{2}s'\) from \(p\), \(\Mass(V_j) \leq C\cdot (s')^{k+1}\) and \(\|S_s\|(B(p,t))\leq \Theta \cdot t^k\) for every \(t\geq 0\) by Remark \ref{Smassbound}, Proposition \ref{reliso} implies that there exists \(V_{j+1}\in \Iak{k+1}(X)\) with \(\spt(S_s-R_s - \partial V_{j+1})\) at distance \(> \frac{1}{2}s' > \frac{1}{4}s_j=s_{j+1}\) from \(p\) and \(\Mass(V_{j+1}) < \left(\frac{3}{32}\right)^{k+1}\cdot \varepsilon\cdot (s')^{k+1} < \frac{\varepsilon}{4^{k+1}}s_{j+1}^{k+1}\).\\

On the other hand, if \(4r \geq s_j\), then 
\[
F_{p,r}(S_s - R_s) \leq \frac{\Mass(V_j)}{r^{k+1}} < \frac{\varepsilon\cdot s_j^{k+1}}{4^{k+1}r^{k+1}} \leq \varepsilon
\]
so we are done.\\

(ii) Let us now assume that we are given a minimizer \(S \in \Zloc{k}(X)\) with \(F_\infty(S - R) = 0\) and \(r\geq r_\varepsilon\). As \(F_\infty(S-R) = 0\), there exists an \(s \geq r\) such that there exists an \(V_0 \in \Iak{k+1}(X)\) with \(\spt(S - R-\partial V) \cap B(p,s) = \emptyset\) and \(\Mass(V_0) < \frac{\varepsilon}{4^{k+1}}\cdot s^{k+1}\). Now we can just repeat the proof of (i) with \(S\) and \(R\) in place of \(S_s\) and \(R_s\), respectively.
\end{proof}

Now we can finish the proof of existence of a solution to the asymptotic Plateau problem for strongly immovable cycles in the Tits boundary in exactly the same way as Kleiner and Lang do in \cite[Theorem 5.6]{kl_hrh}. We give the proof for completeness.

\begin{theorem}\label{fsol}
Let \(X\) be a proper Hadamard space, \(k\geq 1\) an integer and \(\bar{Z}\in \Zak{k-1}(\partial_T X)\) a strongly immovable cycle. Let \(p\in X\) and \(R\in \Zloc{k}(X)\) denote the cone from \(p\) over \(\bar{Z}\). There exists an absolutely minimizing cycle \(S \in \Zloc{k}(X)\) such that \(F_\infty(S-R) = 0\).
\end{theorem}

\begin{proof}
For each \(r > 0\), let \(S_r\in \Iak{k}(X)\) be a minimal filling of \(\partial (R\res B(p,r)) = \langle R,d_p,r\rangle\). By Remark \ref{Smassbound}, for every \(0\leq r \leq s\) it holds that \(\|S_s\|(B(p,r)) \leq \Theta \cdot r^k\) where \(\Theta \coloneqq \Theta_\infty(R)\) and as \(\spt(\partial S_s) \cap B(p,r) = \emptyset\) when \(r < s\), there exists by \cite[Theorem 2.3(2)]{kl_hrh} an absolutely minimizing \(S\in \Iloc{k}(X)\) and a sequence \(s_j \nearrow \infty\) such that \(S_{s_j} \rightarrow S\) as \(j\rightarrow \infty\), in the local flat topology. Let \(\varepsilon > 0\). By Proposition \ref{partial_filling}(i), there exists an \(r_\varepsilon > 0\) such that if \(s\geq r\geq r_\varepsilon\) then \(F_{p,r}(S_s-R_s) < \varepsilon\) and by \cite[Lemma 5.5]{kl_hrh}, \(\lim_{j\rightarrow \infty}F_{p,r}((S-S_{s_j})-(R-R_{s_j})) = 0\) for every \(r > 0\) as \((S-S_{s_j})-(R-R_{s_j})\rightarrow 0\) in the local flat topology as \(j\rightarrow \infty\). Hence for \(r\geq \bar{r}\), it holds that
\[
F_{p,r}(S-R) \leq \limsup_{j\rightarrow \infty} \left(F_{p,r}((S-S_{s_j})-(R-R_{s_j})) + F_{p,r}(S_{s_j} - R_{s_j})\right) < \varepsilon.
\]
This shows that \(F_\infty(S-R) = 0\).
\end{proof}

\begin{remark}
From \cite[Lemma 7.2(3)]{kl_hrh} it follows that if \(\bar{Z}\neq 0\) then \(S\neq 0\).
\end{remark}

Our next goal is to show that \(\partial_T \spt(S) = \spt(\bar{Z})\) when \(S\) and \(\bar{Z}\) are as in the previous theorem. That is done in the same way as in \cite{kl_hrh}: The inclusion \(\partial_T \spt(S)\subseteq \spt(\bar{Z})\) is proven by establishing first the following proposition which is just an adaption of \cite[Theorem 8.2(2)]{kl_hrh} and proven in the same way, and then concluding as in the proof of \cite[Proposition 8.2]{kl_hrh}. The reverse inclusion is proven by establishing Theorem \ref{uniquetangent} which is a substitute for \cite[Theorem 7.3]{kl_hrh} and then concluding as in \cite[Theorem 7.3]{kl_hrh}.

\begin{proposition}\label{sublinear}
\emph{(cf. \cite[Theorem 8.2(2)]{kl_hrh})}
Let \(X\) be a proper Hadamard space, \(k\geq 1\) an integer and \(\bar{Z}\in \Zak{k-1}(\partial_T X)\) a strongly immovable cycle. Let \(p\in X\) and \(R\in \Zloc{k}(X)\) denote the cone over \(\bar{Z}\) from \(p\). For every \(\varepsilon > 0\) there exists an \(r_\varepsilon > 0\) such that if \(S\in \Zloc{k}(X)\) is a minimizer with \(F_\infty(S-R) = 0\) and \(x\in \spt(S)\) is a point with \(d(p,x) \geq r_\varepsilon\) then \(d(x,\spt(R)) < \varepsilon\cdot d(p,x)\).
\end{proposition}

\begin{proposition}\label{minsub}
Let \(X\) be a proper Hadamard space, \(k\geq 1\) an integer and \(\bar{Z}\in \Zak{k-1}(\partial_T X)\) a strongly immovable cycle. Let \(p\in X\) and \(R\in \Zloc{k}(X)\) denote the cone over \(\bar{Z}\) from \(p\). If \(S \in \Zloc{k}(X)\) is a minimizer with \(F_\infty(S-R) = 0\) then \(\partial_T\spt(S)\subseteq \spt(\bar{Z})\).
\end{proposition}

For the reverse inclusion, we need the following simple lemma which is stated in the proof of \cite[Theorem 9.4]{kl_hrh} and whose proof is also outlined there.

\begin{lemma}\label{bplemma}\emph{(cf. proof of \cite[Theorem 9.4]{kl_hrh}).}
Let \(X\) be a proper Hadamard space, \(k\geq 1\) an integer, \(\bar{Z}\in \Zak{k-1}(\partial_T X)\) a cycle, \(p,q\in X\) and \(R_p,R_q\in \Zloc{k}(X)\) the cones over \(\bar{Z}\) from \(p\) and \(q\), respectively. Then \(F_\infty(R_q - R_p) = 0\).
\end{lemma}

We now establish the reverse inclusion as a part of the following theorem which also shows that the solutions to the asymptotic Plateau problem for strongly immovable cycles have a unique tangent cone at infinity (cf. \cite[Proposition 6.10]{hks_morse2}). The proof relies on arguments from the proof of \cite[Theorem 7.3]{kl_hrh} together with the previous proposition and lemma.

\begin{theorem}\label{uniquetangent}
\emph{(cf. \cite[Theorem 7.3]{kl_hrh}, \cite[Proposition 6.10]{hks_morse2})}
Let \(X\) be a proper Hadamard space, \(k\geq 1\) an integer and \(\bar{Z}\in \Zak{k-1}(\partial_TX)\) strongly immovable. For each \(p \in X\), let \(R_p \in \Zloc{k}(X)\) denote the cone over \(\bar{Z}\) from \(p\). Assume that \(S\in \Zloc{k}(X)\) is area minimizing with \(F_\infty(S-R_p) = 0\) for some \(p\in X\). Then for every \(q\in X\) it holds that \((\varrho_{q,\lambda})_\# S \rightarrow R_q\) in the weak topology as \(\lambda \searrow 0\).
\end{theorem}

\begin{proof}
Let \(\Theta \coloneqq \Theta_\infty(S) = \Theta_\infty(R)\). First, as in the proof of \cite[Theorem 7.3]{kl_hrh}, we note that for every \(q\in X\) and every \(0 < \lambda \leq 1\), we have that
\begin{equation}\label{massbound}
\begin{split}
\|(\varrho_{q,\lambda})_\# S\|(B(p,r)) &= \Mass((\varrho_{q,\lambda})_\# (S\res B(p,\frac{1}{\lambda}r))) \leq \lambda^k \cdot \Mass(S \res B(p,\frac{1}{\lambda}r))\\
&\leq \lambda^k \cdot \Theta \cdot (\frac{1}{\lambda}r)^k = \Theta\cdot r^k
\end{split}
\end{equation}
since \(\varrho_{q,\lambda}\) is \(\lambda\)-Lipschitz. Now, let \(S_j\coloneqq (\varrho_{q,\lambda_j})_\# S\) where \(0 < \lambda_j \leq 1\) is a sequence of numbers such that \(\lambda_j \searrow 0\). We will show that \(S_j \rightarrow R_q\) weakly as \(j\rightarrow \infty\). It is shown in the proof of \cite[Theorem 7.3]{kl_hrh}, that we may assume that there exists an \(R'\in \Zloc{k}(X)\) which is conical with respect to \(q\) and such that after possibly passing to a subsequence, \(S_j\rightarrow R'\) in the weak topology as \(j\rightarrow \infty\). The proof also shows that \(F_\infty(S-R') < \infty\) and \(\partial_T\spt(R')\subseteq \partial_T\spt(S)\) so \(\partial_T \spt(R')\subseteq \spt(\bar{Z})\) as \(\partial_T \spt(S) \subseteq \spt(\bar{Z})\) and hence is \(\partial_T\spt(R')\) compact in \(\partial_TX\). Furthermore, from \eqref{massbound}, we know that \(\|R'\|(B(p,r)) \leq \Theta \cdot r^k\) for every \(r\geq 0\). Further, \(F_\infty(R'-R_q) < \infty\) since by Lemma \ref{bplemma}, \(F_\infty(S-R_q) = 0\) so \(F_\infty(R'-R_q) \leq F_\infty(S - R_q) + F_\infty(S-R') < \infty\) as \(F_\infty(S-R') < \infty\). Using that \(F_\infty(R'-R) < \infty\) together with the slicing theorem and the fact that \(R_q\) and \(R'\) have bounded density at infinity, it is easy to show the existence of a constant \(C > 0\) and sequences \(0 < r_j \nearrow \infty\) and \(V_j\in \Iak{k+1}(X)\), \(j\in \N\), such that for every \(j \in \N\) it holds that \(\Mass(V_j) \leq C\cdot r_j^{k+1}\), \(\Mass(\partial V_j) \leq C\cdot r_j^k\) and \(\spt(\partial V_j - (R'-R_q))\subseteq S(p,r_j)\). Let \(\iota_j: (X,d)\rightarrow (X,\frac{1}{r_j}d)\) denote the identity map. By Wenger's compactness theorem \cite[Theorem 1.2]{wenger_compact}, there exists a complete metric space \(Z\) together with isometric embeddings \(\varphi_j:(X,\frac{1}{r_j}d)\rightarrow Z\) and \(\bar{V}\in \Iak{k+1}(Z)\), \(\bar{R}',\bar{R}\in \Iak{k}(Z)\), such that after possibly passing to a subsequence, \((\varphi_j\circ \iota_j)_\# V_j \rightarrow \bar{V}\), \((\varphi_j\circ \iota_j)_\# (R'\res B(p,r_j)) \rightarrow \bar{R}'\) and \((\varphi_j\circ \iota_j)_\# (R_q \res B(p,r_j))\rightarrow \bar{R}\) in the flat topology as \(j\rightarrow \infty\). By Proposition \ref{embedding_lemma}, \(\spt(\bar{V})\) embeds isometrically into the asymptotic cone \(X_\omega \coloneqq \lim_\omega (X,\frac{1}{r_j}d,p)\) where \(\omega\) is some non-principal ultrafilter on \(\N\), so we may consider \(\bar{V}\) as an element of \(\Iak{k+1}(X_\omega)\) and \(\bar{R},\bar{R}'\) as elements of \(\Iak{k}(X_\omega)\). As \(R,R'\) admit lifts to the Tits cone by Theorem \ref{lift}, we can show in exactly the same way as in the proof of Proposition \ref{reliso} that \(\bar{R},\bar{R}'\) are the lifts of \(R,R'\) to the Tits cone, respectively, using the embedding from Lemma \ref{embedding_lemma} to view \(C_TX\) as a subspace of \(X_\omega\). By the condition \(\spt(\partial V_j - R'+R_q) \subseteq S(p,r_j)\) we know that \(\bar{T}\coloneqq \partial \bar{V}-\bar{R}'+\bar{R}\) is supported on \(S(p_\omega,1)\). As \(\Theta_\infty(R')\leq \Theta_\infty(R) = \Theta\), it follows from Theorem \ref{lift} that \(\|\bar{R}'\|(B(p_\omega,r)) \leq \Theta\cdot r^k = \|\bar{R}\|(B(p_\omega,r))\) for every \(0 \leq r \leq 1\). Now, \(\bar{T} + \bar{R}'\) is a filling of \(\partial \bar{R}\) which hence contains \(\bar{R}\) as a slice, as \(\partial \bar{R}\) is strongly immovable. Using that \(\bar{T}\) is supported on \(S(p_\omega,1)\), and that \(\|\bar{R}'\|(B(p_\omega,r)) \leq \|\bar{R}\|(B(p_\omega,r))\) for every \(0\leq r \leq 1\), we can show in exactly the same way as in the proof of Proposition \ref{reliso} above, that \(\bar{R}\res B(p_\omega,r) = \bar{R}'\res B(p_\omega,r)\) for every \(0 \leq r < 1\) and then conclude that \(\bar{R}' = \bar{R}\) so \(R' = R_q\). This finishes the proof.
\end{proof}

\begin{corollary}\label{eq_tits}
\emph{(cf. \cite[Proposition 8.2]{kl_hrh})}
Let \(X\) be a proper Hadamard space, \(k\geq 1\) an integer and \(\bar{Z} \in \Zak{k-1}(\partial_TX)\) strongly immovable. Then for every area minimizing \(S\in \Zloc{k}(X)\) with \(F_\infty(S - R) = 0\) where \(R\in \Zloc{k}(X)\) is a cone over \(\bar{Z}\), it holds that \(\partial_T \spt(S) = \spt(\bar{Z})\).
\end{corollary}

\begin{proof}
By Proposition \ref{minsub}, \(\partial_T\spt(S)\subseteq \spt(\bar{Z})\). As \(\partial_T \spt(S) \supseteq \partial_T \spt((\varrho_{p,\lambda})_\#S)\) for every \(p\in X\) and every \(0 < \lambda \leq 1\), \(\partial_T\spt(S)\supseteq \spt(\bar{Z})\) follows from Theorem \ref{uniquetangent} (cf. \cite[Theorem 7.3]{kl_hrh} and its proof).
\end{proof}

\renewcommand*{\bibfont}{\footnotesize}
\sloppy
\printbibliography

{\footnotesize
{\sc Department of Mathematics, ETH Zürich, Rämistrasse 101, 8092 Zürich, Switzerland}\\
\indent
{\it Email address}: \href{mailto:hjalti.isleifsson@math.ethz.ch}{hjalti.isleifsson@math.ethz.ch}}

\end{document}